\documentclass[12pt,twoside,final,psamsfonts]{amsart}
\usepackage[psamsfonts]{amssymb}
\usepackage{times,a4wide}
\usepackage[dvips,final]{graphicx,psfrag}
\usepackage[draft=false]{hyperref}
\theoremstyle{plain}
\newtheorem*{theorem*}{Theorem}
\newtheorem{theorem}{Theorem}[section]

\newtheorem{proposition}[theorem]{Proposition}
\newtheorem{corollary}[theorem]{Corollary}
\newtheorem{lemma}[theorem]{Lemma}

\theoremstyle{definition}
\newtheorem*{definition*}{Definition}
\newtheorem{definition}[theorem]{Definition}

\newenvironment{proof*}{\vskip 2mm\noindent {}}

\newcommand{\nc}{\newcommand}

\newcommand{\C}{{\mathbb C}}
\newcommand{\N}{{\mathbb N}}
\newcommand{\D}{{\mathbb D}}

\newcommand{\mD}{\mathcal{D}}

\newcommand{\Hol}{\operatorname{Hol}}
\newcommand{\Har}{\operatorname{Har}}

\newcommand{\diz}{{(1-|z|)}}
\newcommand{\dizz}{{1-|z|^2}}

\newcommand{\eit}{{e^{i\theta}}}

\renewcommand{\aa}{{A^{-\alpha}}}
\newcommand{\hid}{\mathcal H^\infty }
\newcommand{\nev}{\mathcal N}
\newcommand{\smi}{\mathcal N_+ }
\nc{\beqa}{\begin{eqnarray*}}
\nc{\eeqa}{\end{eqnarray*}}

\renewcommand{\N}{{\mathcal N}}
\renewcommand{\H}{{\mathcal H}}
\renewcommand{\Re}{\operatorname{Re}}

\author{Xavier Massaneda \& Pascal J. Thomas}

\address{Xavier Massaneda\\ Departament de Matem\`atiques i Inform\`atica,
Universitat  de Barcelona, Gran Via 585, 08007-Bar\-ce\-lo\-na, Catalonia}
\email{xavier.massaneda@ub.edu}

\address{P.J. Thomas\\
Institut de Math\'ematiques de Toulouse; UMR5219 \\
Universit\'e de Toulouse; CNRS \\
UPS, F-31062 Toulouse Cedex 9, France} 
\email{pascal.thomas@math.univ-toulouse.fr}

\date{\today}


\title[From $\H^\infty$ to $\mathcal N$]{From $\H^\infty$ to $\mathcal N$. 
Pointwise properties and algebraic structure\\
in the Nevanlinna class}

\thanks{First author supported by the Generalitat de Catalunya 
(grant 2017 SGR 359) and the Spanish Ministerio de Ciencia,  
Innovaci\'on y Universidades (project MTM2017-83499-P)}

\begin{document}

\begin{abstract} 
This survey shows how, for
the Nevanlinna class $\N$ of the unit disc, 
 one can define and often characterize the analogues of well-known objects and 
properties
related to the algebra of bounded analytic functions
$\H^\infty$: interpolating sequences, Corona theorem, sets of determination,
stable rank, as well as the more recent notions of Weak Embedding Property
and threshold of invertibility for quotient algebras.
The general rule we observe is that a given result for
$\H^\infty$ can be transposed to $\N$ by replacing uniform bounds  
by a suitable control by positive harmonic functions.
We show several instances where this rule applies, as well as some exceptions. 
We 
also briefly discuss the situation for the related Smirnov class. 
\end{abstract}

\maketitle

\section{Introduction}

\label{intro}
\subsection{The Nevanlinna class}
The class $\H^\infty$ of bounded holomorphic functions on the unit $\D$ disc 
enjoys a wealth of analytic and algebraic properties, which have been
explored for a long time with no end in sight, see e.g. \cite{Ga}, 
\cite{Nik02}. 
These last few years, some similar properties have been explored for the 
Nevanlinna class, a much larger algebra which is in some respects a natural 
extension of $\hid$. 
The goal of this paper is to survey those results.  

For $\H^\infty$, and for a function algebra in general, under the heading 
``pointwise properties'' we mean the 
characterization of zero sets, interpolating sequences, and sets of 
determination; while the (related) 
aspects of algebraic structure of $\H^\infty $ we are interested in concern its 
ideals: the Corona theorem,
which says that $\D$ is dense in the maximal ideal space of $\H^\infty $, 
the computation of stable rank, and the determination of invertibility in a 
quotient algebra 
from the values of an equivalence class over the set where they coincide.

According the most common definition, the Nevanlinna class is the algebra of 
analytic functions
\[
 \N=\bigl\{f\in\Hol(\D) : \sup_{r<1} \int_0^{2\pi} \log_+\left| f(r\eit)\right| 
\frac{d\theta}{2\pi} < \infty.\bigr\}.
\]

We will take a different --but, of course, equivalent-- perspective. 
One way to motivate the introduction of this algebra is to consider the 
quotients
$f/g$ when $f, g \in \hid$ and $f/g$ is holomorphic. 
Without loss of generality we may assume $\|f\|_\infty, \|g\|_\infty \le 1$.
Once we factor out the 
common zeroes
of $f$ and $g$ using a Blaschke product (see Section~\ref{preliminaries}), we 
may assume that 
$g(z)\neq 0$ for any $z\in\D$. 
Then $|g|=\exp (-h)$, where $h$ is a positive harmonic function on the disc, and
$\log |f/g| \le h$. We thus make have the following definition, which will 
prove more useful in what follows. 


\begin{definition}
\label{defnev}
A function $f$ holomorphic on $\D$ is in the \emph{Nevanlinna Class}  $\mathcal 
N$, if and only if 
$\log |f|$ admits a positive harmonic majorant.
\end{definition}

We denote by $\Har_+(\D)$ the cone of positive (nonnegative) harmonic 
functions in the unit disc.
For $z\in\mathbb D$ and $\theta\in[0,2\pi)$ let the \emph{Poisson kernel}
\[
P_z(\eit):= \frac{1}{2\pi} \frac{\dizz}{|1-\eit \bar z|^2}.
\]
It is well-known that given $\mu$ a positive finite measure on $\partial \D$, 
its \emph{Poisson integral} 
\[
\mathcal P[\mu](z):= \int_0^{2\pi} P_z(\eit) d\mu(\theta)
\]
belongs to $\Har_+(\mathbb D)$, and reciprocally, any $h\in \Har_+(\mathbb D)$
is the Poisson integral of a positive measure on the circle.
Since any finite real measure is the difference
of two positive measures, 
decomp of measures
and since any harmonic function which admits a positive 
harmonic majorant is 
the difference of two positive harmonic functions,
the Poisson integral of  finite real measures on the circle coincide with 
difference
of positive harmonic functions. This 
explains the equivalence of the definitions above,
and shows that any function in $\N$ is a quotient $f/g$ of bounded analytic 
functions.

The main goal of this survey is to provide, when possible,  analogues for $\N$ 
of several well-known results on 
$\H^\infty$. We shall see that, in order to transfer the results 
we report on from $\H^\infty$ to $\N$, we must apply the following general 
principle. A function $f$ belongs to 
$\H^\infty$ if and only if $\log_+|f|$ is uniformly bounded above, while 
$f\in\N$ 
when  there exists $h\in\Har_+(\D)$ such that $\log_+|f|\leq h$. Accordingly, 
in the hypotheses (and sometimes the conclusions) of the theorems regarding 
$\H^\infty$
we expect to replace uniform bounds by positive harmonic majorants.  This turns 
out to yield very 
natural results and, sometimes, natural problems.

A function with a positive harmonic majorant is harder to grasp intuitively
than a bounded function.  Admitting a positive harmonic 
majorant is definitely a restriction; for instance, Harnack's inequalities 
\eqref{harnack}
show that a positive harmonic function cannot grow faster that $C\diz^{-1}$ as
$|z|\to 1$.  But there is no easy way to recognize whether a given nonnegative
function on the disc has a harmonic majorant, although some conditions are
given in \cite{HMNT}, \cite{BNT}. 

To give a very simple example, suppose that
we are given a sequence $(z_k)_k \subset \D$ and positive numbers $(v_k)_k$.
Define a function $\varphi$ on $\D$ by 
$\varphi(z_k)=v_k$, $\varphi(z)=0$
when $z\notin (z_k)_k $. It follows from Harnack's inequality that a necessary 
condition for $\varphi$ to admit
a harmonic majorant is $\sup_k (1-|z_k|) v_k < \infty$. On the other hand, a 
sufficient condition is that
$\sum_k (1-|z_k|) v_k < \infty$, and the gap between the two conditions cannot 
be improved without additional 
information about the geometry of the sequence $(z_k)_k$ (see 
Corollary~\ref{interpolation-geometric} in
Section~\ref{interpolation}).

The paper is structured as follows. In the next Section we finish describing 
the 
setup and gather several well-known properties of Nevanlinna functions. 
In particular, we recall that from the point of view of the topology $\N$ is 
far worse than $\H^\infty$. Even though there is a well-defined distance $d$ 
that makes 
$(\N,d)$ a metric space, this is not even a topological vector space. 
This is an important restriction, since all the nice theorems about Banach 
spaces available for $\H^\infty$ are no longer available. The canonical 
factorization of Nevanlinna functions, 
similar to that of bounded analytic functions, makes up for these shortcomings. 

Section~\ref{interpolation} studies the analogue in $\N$ of the interpolation 
problem  
in $\H^{\infty}$, which was completely solved in a famous theorem of Carleson 
(see e.g. \cite{Ga}). 
The first issue is to give the appropriate definition of Nevanlinna 
interpolation. Once this is 
done we show that the general principle described above, applied to the 
various characterizations of $\H^\infty$-interpolating sequences, provides also 
 
characterizations of Nevanlinna interpolating sequences.
The same principle works well when describing finite unions 
of interpolating sequences.

An old theorem of Brown, Shields and Zeller \cite{BrShZe} shows that the 
$\H^\infty$-sampling sequences, i.e. sequences $(z_k)_k\subset \mathbb D$ such 
that
\[
 \|f\|_\infty=\sup_k |f(z_k)|\quad \textrm{for all $f\in \H^\infty$},
\]
are precisely those for which the non-tangential accumulation set of $(z_k)_k$ 
on $\partial\D$ has full measure. 
In Section~\ref{sampling} we set the analogous problem for $\N$, show the 
complete solution given by S. Gardiner, 
and provide some examples obtained by previous, more computable, conditions 
given by the authors.

R. Mortini showed that the corona problem in $\N$ can be solved analogously to 
Carleson's classical result for $\H^\infty$, provided that the substitution 
suggested by our guiding principle is performed.
In Section~\ref{ideals} we show this and a result about finitely generated 
ideals, which is the natural counterpart of
a theorem by Tolokonnikov \cite{Tol83}. We also show that, in contrast to the 
$\H^\infty$-case, 
the stable rank of the algebra $\N$ has to be strictly bigger that $1$.

Section~\ref{wep} studies the Corona problem in quotient algebras.  Gorkin, 
Mortini and 
Nikolski \cite{GMN2} showed that that $\H^\infty$ quotiented by an inner 
function has the 
Corona property if and only if it has the so-called Weak Embedding Property 
(essentially, the inner function is uniformly big away, by a fixed distance, 
from its zeroes; see Theorem~\ref{bounded-wep}(c)). 
The analogue is also true for $\N$, but unlike in the bounded case, this is 
equivalent to the inner function being a Blaschke product of a finite number of 
Nevanlinna interpolating sequences. We are also interested in the determination 
of 
invertibility in a quotient algebra from the values of an equivalence class 
over the 
set where they coincide.

The final Section~\ref{smirnov} is devoted to describe the results for the 
Smirnov 
class $\N_+$, the subalgebra of $\N$ consisting of the functions $f$ for which
belongs to the Smirnov 
$\log_+|f|$ has a \emph{quasi-bounded} harmonic majorant, i.e. a 
majorant of type $\mathcal P[w]$, where $w\in L^1(\partial\mathbb D)$. A 
posteriori $\log|f(z)|\leq\mathcal P[\log |f^*|](z)$, $z\in\mathbb D$. The 
general rule,
with few exceptions, is that the statements about $\N$ also hold for $\N_+$ as 
soon as
the bounded harmonic majorants are replaced by quasi-bounded harmonic majorants.

A final word about notation. Throughout the paper  $A\lesssim B$ will mean that 
there is an absolute constant $C$ such that $A \le CB$, and we write $A\approx 
B$ 
 if both $A\lesssim B$ and $B\lesssim A$.

\section{Preliminaries}\label{preliminaries}

It has been known for a long time that functions in the Nevanlinna class
have the same zeroes as bounded functions. These are the sequencies 
$Z:=(z_k)_{k \in \mathbb N} \subset 
\D$ satisfying the \emph{Blaschke condition}  $\sum_k (1-|z_k|) <\infty$. 
Note that when points in the sequence $(z_k)_k$ are repeated, we understand that
$f$ must vanish with the corresponding order. 

That this is necessary is an easy application of Jensen's formula. 
To prove the reverse implication just notice that
when $(z_k)_k$ is a Blaschke sequence the 
the associated \emph{Blaschke product} 
$$
B_Z (z) := \prod_{k\in \mathbb N} \frac{\bar z_k}{|z_k|}\frac{z_k-z}{1-z\bar 
z_k|},
$$
converges and yields a holomorphic function with 
non-tangential limits of modulus $1$
almost everywhere on the unit circle $\partial \D$.
We simply write $B(z)$ when no confusion can arise.

\begin{theorem}\cite[Lemma 5.2, p. 69]{Ga} \label{zeroes}
Let $f\in \nev$, $f\not \equiv 0$, $Z:=f^{-1}\{0\}$. Then $B_Z$ converges,
and $g:=f/B_Z\in\nev$. Moreover, $\log |g|$ is the least harmonic majorant
of $\log|f|$.
\end{theorem}

Since $\N$ is the set of quotients $f/g$, where $f,g\in \H^\infty$ and $g$ is 
zero free, an
easy corollary is that $F\in \nev$ admits an inverse in $\nev$ if and
only if it does not vanish on $\D$.

Any $f\in\hid$ verifies $\log|f(z)|\le \mathcal P[\log|f^*|](z)$, where we 
recall that $f^*$
denotes the nontangential boundary values of $f$. Therefore if $F=f/g\in\N$, 
$F\not \equiv0$, 
$F^*(\eit)\neq 0$ almost everywhere, it admits finite nontangential boundary 
values almost everywhere, again denoted $F^*=(f/g)^*$. We will see more on this 
in Section~\ref{factorization}.

Throughout the paper we will find it useful to consider the product 
with one factor removed and write 
\[
b_{z_k}(z)= \frac{\bar z_k}{|z_k|}\frac{z_k-z}{1-z\bar z_k}\quad,\quad 
B_{k}(z):=B_{Z,k} (z) :=\frac{B_Z(z)}{b_{z_k}(z)}= \prod_{j\in \mathbb N, j\neq 
k} \frac{\bar 
z_j}{|z_j|}\frac{z_j-z}{1-z\bar z_j}.
\]

We shall often need to use the \emph{pseudohyperbolic} or \emph{Gleason} 
distance between points of the disc given by
\[
\rho(w,z):= \left| \frac{w-z}{1-z\bar w}\right|.
\]
It is invariant under automorphisms (holomorphic bijections) of the disc, and
closely related to the Poincar\'e distance. 
We shall denote by $D(a,r)$ the disk centered at $a\in\D$ and with radius $r\in 
(0,1)$, 
with respect to $\rho$, that is,
$D(a,r)=\{ z\in\D : \rho(z,a)<r\}$.

Regarding basic properties of harmonic functions, 
we will use repeatedly the well-known \emph{Harnack inequalities}: for 
$H\in\Har_+(\mathbb D)$ and $z,w\in\mathbb \D$,
\[ 
\frac{1-\rho(z,w)}{1+\rho(z,w)}\leq\frac{H(z)}{H(w)}\leq\frac{1+\rho(z,w)}{
1-\rho(z,w)}\ .
\]
In particular, taking $w=0$,
\begin{equation}\label{harnack}
H(0)\frac{1-|z|}{1+|z|}\leq H(z)\leq H(0) \frac{1+|z|}{1-|z|}º .
\end{equation}

We now collect some standard facts about the Nevanlinna class. We start with 
the natural metric
and the topolgy it defines. We show next the canonical factorization of 
Nevanlinna functions.

\subsection{Topological properties of $\N$}\label{topology}
From the point of view of the topology, $\N$ is far worse than $\H^\infty$.
One easily sees that, given $f\in\Hol(\D)$, 
 $\log(1+ |f|)$ admits a harmonic majorant if and only if $\log_+|f|$ does. 
Thus we may define a distance between $f$ and $g$ in the Nevanlinna class by 
$d(f,g)=N(f-g)$, where
\begin{equation}
\label{ndef}
N(f):=\lim_{r \to 1}\frac{1}{2\pi}\int_{0}^{2\pi}
     \log(1+|f(re^{i\theta})|)\;d\theta\ .
\end{equation}
The subharmonicity of $\log (1+|f|)$ yields the pointwise estimate
\[
(1-|z|)\log (1+|f(z)|)\leq 2N(f),
\]
which shows that convergence in the distance $d$ implies 
uniform convergence on compact sets \cite[Proposition 1.1]{SS}.

The value $N(f)$ can also be rewritten as an extremal solution to a harmonic 
majorant problem:  
\begin{equation}
\label{ndef2}
N(f)=\inf\bigl\{h(0)\, :\,\exists h\in\Har_+(\D)\mbox{ with }\log(1+|f|)\leq h 
\bigr\}\ .
\end{equation}

Although $(\N,d)$ is a metric space, it is not a topological vector space: 
multiplication by scalars fails to be continuous, and it contains many
finite-dimensional subspaces on which the induced topology is discrete. 
In fact, the largest topological vector space contained in $\nev$
is the Smirnov class $\N_+$ (see Section 8). These facts and many others are 
proved in 
\cite{SS}.

The limitations imposed by this lack of structure are important, for example, 
when studying interpolating or 
sampling
sequences (see Sections~\ref{interpolation} and \ref{sampling}), since such 
basic tools as the Open Mapping 
or the Closed Graph theorems are not available.

\subsection{Canonical factorization}\label{factorization}
To compensate for the lack of structure just mentioned, there is a canonical 
factorization of functions  both in $\H^\infty$ and $\N$ (general refe\-rences 
are 
e.g. \cite{Ga}, \cite{Nik02} or \cite{RR}). 
As a matter of fact this is at the core of our transit from $\H^\infty$ to $\N$,
although quite often in the process the existing proofs for $\H^\infty$ have to 
be redone.

As mentioned in the previous section, 
given $f\in\N$ with zero set $Z$ and  associated Blaschke product $B_Z$, 
the function $f/B_Z$ is zero-free and belongs to $\N$ as well 
(Theorem~\ref{zeroes}). 

A function $f$ is called \emph{outer} if it can be written in the form
\[
\mathcal O(z)=C \exp \left\{\int_0^{2\pi} \frac{\eit+z}{\eit-z}
\log v(\eit) d\sigma(\theta) \right\},
\]
where $|C|=1$, $v>0$  a.e.\ on $\partial\mathbb D$ and $\log v \in 
L^1(\partial\mathbb D)$. Such a
function is
the quotient $\mathcal O=\mathcal O_1/\mathcal O_2$ of two bounded outer 
functions  $\mathcal O_1,\mathcal O_2\in \H^\infty$
with  $\|\mathcal O_i\|_{\infty}\le 1$, $i=1,2$. In particular, the weight $v$
is given by the boundary values of $|\mathcal O_1/\mathcal O_2|$. Setting 
$w=\log v$,
we have
\[
\log |\mathcal O(z)|=P[w](z) =\int_0^{2\pi}  P_z(\eit) w(\eit)d\sigma(\eit).
\]
This formula allows us to freely switch between assertions about outer
functions $f$ and the associated measures $w d\sigma$.

Another important family in this context are  \emph{inner} functions: $I\in
\H^\infty$ such that $|I|=1$ almost everywhere on $\partial\D$. Any inner 
function $I$ 
can be factorized into a Blaschke product $B_Z$ carrying the zeros
$Z=(z_k)_k$  of $I$, and a singular inner function $S$ defined by
\[
      S(z)=\exp\left\{-\int_0^{2\pi} \frac{\eit+z}{\eit-z}\,d\mu(\eit)\right\},
\]
for some positive Borel measure $\mu$ singular with respect to the Lebesgue
measure.

According to the Riesz-Smirnov factorization, any
function $f\in \N$ is represented as
\begin{equation}\label{factor}
      f=\alpha \frac{B S_1 \mathcal O_1}{S_2 \mathcal O_2},
\end{equation}
where $\mathcal O_1, \mathcal O_2$ are outer with $\|\mathcal 
O_1\|_{\infty},\|\mathcal O_2\|_{\infty}\le 1$,
$S_1, S_2$ are singular inner, $B$ is a Blaschke product and $|\alpha|=1$.

Similarly, any $f\in \H^\infty$ with zero set $Z$ can be factored as $f=B_Z 
f_1$, 
with $f_1\in \H^\infty$ and $\|f_1\|_\infty=\|f\|_\infty$. Actually $f_1$ takes 
the 
form $f_1=\alpha S\mathcal O$, 
where $\mathcal O$ is outer and bounded, $S$ is singular inner and 
$|\alpha|=1$. In this sense $\H^\infty$ and $\N$ are comparable.

\section{Nevanlinna Interpolation}\label{interpolation}

The first difficulty when trying to study interpolating sequences for $\nev$ is 
to find out what the precise problem should be.  
An interpolating sequence is one over which a 
certain natural set of sequences
of values may be realized as the restriction of functions in the class being 
studied.
That bounded functions should be required to interpolate all bounded
values seems natural.

\begin{definition}
\label{interp}
We say that $(z_k)_{k \in \mathbb N} \subset \D$ is an \emph{interpolating 
sequence} 
for $\H^\infty$
if for any bounded complex sequence $v:= (v_k)_k \in \ell^\infty$, there exists 
$f \in \H^\infty$
such that $f(z_k)=v_k$ for all $k \in \mathbb N$.
\end{definition}

By an application of the Open Mapping theorem, one can show that 
when $(z_k)_k$ is $\H^\infty$-interpolating, 
there exists $M>0$ such that $f$ can be 
chosen with $\|f\|_\infty \le M \| v \|_\infty$. The minimum such $M$ is called 
the \emph{interpolation constant} of $(z_k)_k$ .

The interpolating condition means that point evaluations at each $z_k$ are 
independent of each other in a strong sense, and so the sequence has to be 
sparse.
Actually, it is easy to deduce from
the Schwarz-Pick Lemma that any interpolating
sequence must be \emph{separated}, i.e. $\inf\limits_{j\neq k} \rho(z_j,z_k)>0$.

Actually, a stronger uniform separation is necessary and sufficient.

\begin{theorem}[Carleson, 1958, see e.g. \cite{Ga}]
\label{carlint}
The sequence $(z_k)_k$ is interpolating for $\H^\infty$ if and only if 
\[
\inf_{k\in\mathbb N} \left| B_{k} (z_k)\right| = \inf_{k\in\mathbb N} \prod_{j, 
j\neq k} \rho(z_j,z_k) >0.
\]
\end{theorem}
The $\hid$-interpolating sequences, also known as Carleson sequences, have 
several equivalent
characterizations, as we shall see soon.

Turning to the Nevanlinna class, we see that from the usual definition given 
at the beginning of the paper (in terms of 
limits of integrals over circles),
it is not immediately obvious what the natural condition on pointwise values of 
$f$ should be. 

One way to circumvent this is to forego an explicit description of our target 
set of values, but to demand
that it satisfies some property that guarantees that, in a sense, values can be 
chosen independently over
the points of the sequence. Following N. K. Nikolski and his collaborators, we 
adopt the following definition.

\begin{definition}
\label{freeint} 
Let $X$ be a space of holomorphic functions
in $\D$. A sequence $Z\subset \D$ is called \emph{free interpolating} for
$X$ if the space of restrictions $X|Z$ of functions of $X$ to $Z$,
also called \emph{trace space}, $X\vert Z$, is ideal. This means that 
if $(v_k)_k \in  X| Z$ and $(c_k)_k \in 
\ell^\infty$, then necessarily $(c_k v_k)_k \in X|Z$.
\end{definition}

In other words, $Z$ being free interpolating means that whenever a sequence of  
values is a restriction
of an element in $X$, any other sequence with pointwise values lesser or equal 
in moduli will also be a restriction.
We also see that, when $X$ is stable under multiplication by bounded functions,
 this immediately implies that $Z$ is a zero set for $X$.
 
Observe also that for $X$ a unitary algebra, $X\vert Z$ is ideal if and only if 
$\ell^\infty \subset X\vert Z$.
Indeed, if $X\vert Z$ is ideal, then since the constant function $1$ belongs to 
$X$, any 
bounded sequence must
be the restriction of a function in $X$. Conversely, suppose that $|w_k| \le 
|v_k|$ for each $k$
and that $v_k=f(z_k)$ for some $f\in X$. 
Let $h\in X$ satisfy $h(z_k) = w_k/v_k$ when $v_k\neq 0$; then $(hf)(z_k)=w_k$ 
and $hf\in X$. 

As a consequence, we remark that any $\hid$-interpolating sequence must be free 
interpolating for $\nev$, since any bounded sequence of values will be the 
restriction of a 
bounded function, which is in the Nevanlinna class.

An alternative approach, when trying to define Nevanlinna interpolation, 
is to think in terms of harmonic 
majorants and consider,
for a sequence $Z=(z_k)_k$, the subspace of values
\[
\ell_{\N}(Z):=\left\{ (v_{k})_k: \ \exists\
h\in\Har_+(\D)\
\mbox{ such that }h(z_k)\ge \log_+|v_{k}|, k\in\mathbb N \right\}.
\]
Notice that it is immediate from the first property in Definition~\ref{defnev} 
that 
$\nev|Z \subset \ell_{\N}(Z)$.  We then want to claim that a sequence is 
interpolating
for $\nev$ when the reverse inclusion holds.  Since $\ell_{\N}(Z)$ is clearly 
ideal, this
implies in particular that the sequence is free interpolating. 

Conversely, assume $Z$ is free interpolating.  
Given values $(v_k)_k$ and $h$ as above, we can construct a holomorphic
function $H$ such that $\Re H=h$; then $e^H \in \nev$.  The values 
$(e^{-H(z_k)}v_k)_k$
form a bounded sequence, so by the assumption there is $g\in \nev$ such that 
$g(z_k)= e^{-H(z_k)}v_k$,
and $e^H g \in \nev$ will interpolate the given values. So  the two conditions 
above turn out to be equivalent.

\subsection{Main result and consequences.} 

In this section we report mostly on results from \cite{HMNT}.

Nevanlinna interpolating sequences are characterized by the following
Carleson type condition. 

\begin{theorem}{\cite[Theorem 1.2]{HMNT}}
\label{nevinter}
Let $Z=(z_k)_k$ be a sequence in $\D$. The following statements are
equivalent:
\begin{itemize}
\item[(a)] $Z$ is a free interpolating sequence for the Nevanlinna
class $\nev$, i.e. the trace space $\nev|Z$ is ideal;

\item[(b)] The trace space $\nev|Z$ contains $\ell_{\N}(Z)$, and therefore is 
equal to it;

\item[(c)] There exists $h\in\Har_+(\D)$ such that 
\[
|B_{k}(z_k)|\geq e^{-h(z_k)}\quad k\in\mathbb N.
\]
\end{itemize}
\end{theorem}

Notice that (c) could be rephrased as
\begin{itemize}
 \item [(c)] The function $\varphi_Z$ defined by 
 $\varphi_Z( z_k)= \log  | B_{k} (z_k)|^{-1}$,
$\varphi_Z(z)=0$ when $z\in \D\setminus (z_k)_k$, admits a harmonic majorant.
\end{itemize}
In this form we see that this is the adaptation, according to our general 
principle,
of Carleson's condition for $\H^\infty$: the uniform upper bound on $\log | 
B_{k} 
(z_k)|^{-1}$ is replaced by a harmonic majorant.

A first consequence of the theorem is that Nevanlinna-interpolating sequences 
must 
satisfy a weak separation condition.
\begin{definition}\label{weaksep}
A sequence $Z=(z_k)_k$ is \emph{weakly separated} if there exists 
$H\in\Har_+(\D)$
such that the disks $D(z_k, e^{-H(z_k)})$, $k\in \mathbb N$, are pairwise 
disjoint.
\end{definition}
Since for any $j \neq k$, $\rho(z_j,z_k) \ge \left| B_{k} (z_k)\right|$,
we have $\rho(z_j,z_k) \ge e^{-h(z_k)}$, with $h$ the harmonic majorant in 
condition (c).
So any Nevanlinna-interpolating sequence is weakly separated.
Another proof of this is given in \cite[Corollary 2.5]{HMN2}.

Although it is not easy to deduce simple geometric properties from condition 
(c), there is a vast and interesting class of sequences which turns out to be 
Nevanlinna-interpolating.

In order to see a first family of examples we point out that condition (c) 
is really about the ``local'' Blaschke product; it depends only on the 
behavior of the $z_j$, $j\neq k$,  in a fixed pseudohyperbolic neighbourhood of 
$z_k$.

\begin{proposition}{\cite[Proposition 4.1]{HMNT}}
\label{propsep}
Let $Z$ be a Blaschke sequence.  For any $\delta \in (0,1)$, there
exists a positive harmonic function $h$,
such that
\[
    \log \prod_{{k:\rho(z_k,z)\ge \delta}} |\rho(z_k,z)|^{-1}
    \leq h(z) ,\quad
    z \in \D .
\]
\end{proposition}

The idea of the proof is that for any $a\in \D$ 
and $\delta \in (0,1)$,
there is a constant $C$ such that for $\rho(a,z) \ge \delta$,
\[
\log \frac1{\rho(a,z)} \le C  \mathcal P [  \chi_{I_{a}}] (z),
\]
where 
\[
I_a=\Bigl\{\eit\in\partial\D : \bigl|\eit-\frac a{|a|}\bigr|< 1-|a|\Bigr\}
\]
is the Privalov shadow of $a$ 
(see e.g. \cite[p. 124, lines 3 to 17]{NPT}). One gets then
the result by summing over the sequence $Z$, which gives the Poisson 
integral of an integrable function on the
circle, because of the Blaschke condition.

An immediate consequence is the following.

\begin{corollary}
Any separated Blaschke  sequence is Nevanlinna-interpolating.
\end{corollary}

There are other cases, when the geometry of $Z$ is especially regular, where
it is possible to characterize Nevanlinna interpolating sequences.
This is the case for well concentrated sequences, in the sense that they are 
contained in a finite union of Stolz angles $\cup_i \Gamma(\theta_i)$, 
where 
\[
\Gamma(\theta)=\bigl\{z\in\D : |z-\eit|\leq  (1-|z|)\bigr\}.
\]
It is also the case for for well spread sequences, in the sense that the measure
$\mu_Z=\sum_k (1-|z_k|)\delta_{z_k}$ has bounded Poisson balayage, i.e., 
\[
 \sup_{\theta\in[0,2\pi)} \sum_k (1-|z_k|) P_{z_k}(\eit) \approx
 \sup_{\theta\in[0,2\pi)} \sum_k \frac{(1-|z_k|)^2}{|z_k-\eit|^2}<+\infty.
\]

\begin{corollary}\label{interpolation-geometric}
Assume $Z=(z_k)_k$ is a Blaschke sequence in $\D$. 
\begin{itemize}
 \item [(a)] Let $Z$ be contained in a finite number of Stolz angles. Then $Z$ 
is Nevanlinna 
 interpolating if and only if 
 \begin{equation}\label{blabdd}
  \sup_k (1-|z_k|) \log | B_{k} (z_k)|^{-1}<+\infty.
 \end{equation}
 
 \item[(b)] Let $Z$ be such that $\mu_Z:=\sum_k (1-|z_k|)\delta_{z_k}$ has 
bounded Poisson balayage.
 Then $Z$ is Nevanlinna 
 interpolating if and only if 
 \begin{equation}\label{bal}
  \sum_k (1-|z_k|) \log | B_{k} (z_k)|^{-1}<+\infty.
 \end{equation}
\end{itemize}
\end{corollary}

Condition \eqref{blabdd} is always necessary for Nevanlinna interpolation. 
This is just a consequence of Theorem~\ref{nevinter} (c) and Harnack's 
inequalities (see \eqref{harnack}).

In order to see the converse we can assume that the sequence is contained in 
just one
Stolz angle: if $Z=\cup_{i=1}^n Z_i$, with $Z_i\subset \Gamma(\theta_i)$, 
$\theta_i\neq\theta_j$, then
\[
 \lim_{\begin{subarray} z z\to e^{i\theta_i} \\ z\in \Gamma(\theta_i) 
\end{subarray}}
 |B_{Z_i}(z)|= 1\ ,
\]
and therefore, for $z_k\in\Gamma(\theta_i)$, the value $\log|B_{k}(z_k)|^{-1}$ 
behaves asymptotically like 
$\log|B_{Z_i,k}(z_k)|^{-1}$.

In this situation the proof is immediate. Assume $Z$ is contained in the Stolz 
angle of vertex $1\in\partial\D$.
Let $C$ denote the supremum in \eqref{blabdd} and define the positive harmonic 
(singular) function 
\[
h(z):=P[C\delta_1](z)=C\frac{1-|z|^2}{|1-z|^2}, 
\]
where $\delta_1$ indicates the Dirac mass on $1$. From the hypothesis
\[
 \log|B_{k}(z_k)|^{-1}\leq\frac C{1-|z_k|}\leq h(z_k), \quad k\in \mathbb N,
\]
and the result follows from Theorem~\ref{nevinter}.

On the other hand, condition \eqref{bal} is always sufficient, since then 
\[
w(\theta)=\sum_k (\log | B_{k} (z_k)|^{-1}) \chi_{I_k}(\eit)
\]
is in $L^1(\partial\D)$ and clearly
\[
 \log | B_{k} (z_k)|^{-1}\leq P[w](z_k), \quad k\in\mathbb N.
\]
In order to see that in case the Poisson balayage is finite \eqref{bal}
it is also necessary, take $h\in\Har_+(\D)$
satisfying Theorem~\ref{nevinter}(c) and let $\nu$ be a finite positive measure 
with $h=P[\nu]$. 
Then, by Fubini
\begin{align*}
 \sum_k (1-|z_k|) h(z_k)&=\int_{\D} (1-|z|) h(z) d\mu_Z(z)\\
 &= 
 \int_0^{2\pi}\int_{\D} (1-|z|) P_z(\eit) d\mu_Z(z) d\nu(\theta)\approx 
\nu(\partial\D),
\end{align*}
and the result follows from the previous estimate.

\subsection{Comparison with previous results.}
There had been previous works on the question of interpolation in the 
Nevanlinna class. As early as 1956, Naftalevi\v c \cite{Na56} described the
sequences $\Lambda$ for which the trace $N\vert \Lambda$ coincides with the
sequence space
\[ 
\ell_{\text{Na}}(Z):=\bigl\{(v_k)_k: \sup_{k} (1-|z_k|)\log_+|v_k|<\infty 
\bigr\}.
\]

\begin{theorem}{\cite{Na56} }
$\nev \vert Z = \ell_{\text{Na}}(Z)$ if and
only if $Z$ is contained in a finite union of Stolz angles and \eqref{blabdd} 
holds.
\end{theorem}

On the other hand, in a paper about the Smirnov class, Yanagihara \cite{yana2} 
had introduced  the sequence space 
\[
\ell_{\text{Ya}}(Z):=\bigl\{(z_k)_k:\sum_k (1-|z_k|) \log_+|v_k|<\infty\bigr\}.
\]
As we have just seen in the previous section, for any $Z \subset 
\D$,
$\ell^\infty \subset \ell_{\text{Ya}}(Z) \subset \ell_{\N}(Z) \subset 
\ell_{\text{Na}}(Z)$.

The target space $\ell_{\text{Na}}$ seems ``too big'', since
the growth condition it imposes
forces the sequences to be confined in a finite union of Stolz angles. 
Consequently a big class of $\hid$-interpolating sequences,
namely those containing a subsequence tending
tangentially to the boundary,  cannot be interpolating  in the sense of
Naftalevi\v c. This does not seem  natural, for $\hid$ is in the multiplier
space of $\nev$.

On the other hand, the target space $\ell_{\text{Ya}}(Z)$ seems ``too small'':
there are  $\hid$-interpolating sequences such that $\N|Z$ (or even $\smi\vert 
Z$)
does not embed into  $l_{\text{Ya}}(Z)$ \cite[Theorem 3]{yana2}. 

If one requires that $\nev\vert Z\supset 
\ell_{\text{Ya}}(Z)$,
this implies that all bounded values can be interpolated, and 
the sequence is Nevanlinna-interpolating. But then the natural
space of restrictions of functions in $\nev$ is the a priori larger 
$\ell_{\N}(Z)$. As seen in Corollary~\ref{interpolation-geometric}, 
the target space $\ell_{\text{Ya}}(Z)$ is only natural when considering 
sequences
for which $\mu_Z:= \sum_k (1-|z_k|) \delta_{z_k}$ has bounded Poisson balayage.

\subsection{Equivalent conditions for Nevanlinna 
interpolation}\label{int-equivalent}

The following result collects several alternative descriptions of Nevanlinna 
interpolating sequences. All of them have their corresponding analogues in 
$\H^\infty$.

Given $H\in\Har_+(\D)$, consider the  disks 
$\mD_k^H=D(z_k, e^{-H(z_k)})$ and the domain
\[
 \Omega_k^H=\D\setminus\bigcup_{\stackrel{j:j\neq 
k}{\rho(z_j,z_k)\le 1/2}}
  \mD_j^H.
\]
The proof of Theorem~\ref{alt-interpolation} 
below shows clearly that the choice of 
the constant 1/2 in the definition of $\Omega_k^H$ is of no relevance; it can 
be 
replaced by any $c\in (0,1)$. Let $\omega(z, E,\Omega)$ denote the harmonic 
measure at $z\in\Omega$ of
the set $E\subset \partial \Omega$ in the domain $\Omega$.

\begin{theorem}(\cite[Theorem 1.2]{HMN1})  \label{alt-interpolation}
Let $Z=(z_k)_k$ be a Blaschke sequence of distinct points in $\D$ 
and let $B$ be the Blaschke
product with zero set $Z$. The following statements are equivalent:
\begin{itemize}
\item [(a)] $Z$ is an interpolating sequence for $\N$, that is, there 
exists $ H\in \Har_+(\D)$ such that
\begin{equation*}
 (1-|z_k|^2)|B'(z_k)|=|B_k(z_k)|\ge e^{-H(z_k)},\quad 
k\in\mathbb N.
\end{equation*}
\item [(b)]  There exists $ H\in \Har_+(\D)$ such that $|B(z)|\ge 
e^{-H(z)}\rho(z,Z)$, $z\in\D$,
\item [(c)] There exists $ H\in \Har_+(\D)$ such that 
$|B(z)|+(1-|z|^2)|B'(z)|\ge e^{-H(z)}$, $z\in\D$,
\item [(d)] There exists $ H\in \Har_+(\D)$ such that the disks $ \mD_k^H$ 
are pairwise disjoint, and 
\[
 \inf_{k\in\mathbb N} \omega(z_k,\partial\D,\Omega_k^H)>0.
\]
\end{itemize}
\end{theorem}

The proof of (d) shows that it can be replaced by an a priori 
stronger statement: for every $\epsilon\in(0,1)$ there exists $ 
H\in \Har_+(\D)$ such that the disks $ \mD_k^H$ 
are pairwise disjoint, and 
\[
 \inf_{k\in\mathbb N} \omega(z_k,\partial\D,\Omega_k^H)\geq 1-\epsilon.
\]

Vasyunin proved in \cite{Vas78} (see also \cite{K-L}) that $B_Z$ is an 
$\H^\infty$ interpolating 
Blaschke product if and only if there exists $\delta>0$ such that
\begin{equation}\label{vasyu}
 |B_Z(z)|\geq \delta \rho(z,Z),\quad z\in\D.
\end{equation}
Therefore, condition (b) is, again, the natural counterpart of the 
existing condition for $\H^\infty$.

Similarly, statement (d) and its proof are modelled after the corresponding 
version for 
$\H^{\infty}$, proved by J.B. Garnett, F.W. Gehring and P.W. Jones in 
\cite{GGJ}. 
In that case the pseudohyperbolic discs $\mD_k^H=D(z_k, e^{-H(z_k)})$ have to 
be replaced
by uniform discs $D(z_k,\delta)$, $\delta>0$.

A useful, and natural, consequence of Theorem~\ref{alt-interpolation}(d) is 
that Nevanlinna interpolating sequences are stable under 
small pseudohyperbolic perturbations.

\begin{corollary}\label{stability}
 Let $Z = (z_k)_k$ be a Nevanlinna interpolating sequence and let 
$H\in\Har_+(\D)$,  satisfying Theorem~\ref{alt-interpolation}(a).  If 
$Z'=(z_k^\prime)_k \subset\mathbb D$ satisfies 
 \[
  \rho(z_k, z_k^\prime )\leq \frac 14 e^{-H(z_k)}\ ,\quad k\in\mathbb 
N,
 \]
 then $Z'$ is also a Nevanlinna interpolating sequence.
\end{corollary}

\subsection{Peak functions and interpolation}

When $Z=(z_k)_k$ is $\H^\infty$-interpolating, an application of the Open 
Mapping Theorem to the restriction operator $\mathcal R(f)=(f(z_k))_k$ shows 
that there exist $C>0$ and functions $f_k\in \H^\infty$, $k\in\mathbb N$, such 
that $\|f_k\|_\infty\leq C$ and
\begin{equation}\label{1i0}
 f_k(z_j)=
 \begin{cases}
  1\ &\textrm{if $j=k$}\\
  0 &\textrm{if $j\neq k$}.
 \end{cases}
\end{equation}
As it turns out, the existence of such peak functions in fact implies that $Z$ 
is $\H^\infty$-interpolating (see \cite[Chap. VII]{Ga}).

For the Nevanlinna class, which, as mentioned, has a much weaker structure than 
$\H^\infty$, the analogous result does not hold.
On the one hand, the proof of Theorem~\ref{nevinter} shows that when $Z$ is 
Nevanlinna interpolating
there exist $C>0$ and $f_k\in \N$, $k\in\mathbb N$, such that 
$N(f_k)\leq C$ and \eqref{1i0} holds. Notice that no Open Mapping Theorem can be 
applied to deduce this.

But the converse fails: there are examples of sequences $Z$ for which there exist $C>0$ and 
$f_k\in \N$, $k\in\mathbb N$, with
$N(f_k)\leq C$, satisfying \eqref{1i0}, but which are not Nevanlinna interpolating (see 
\cite[Theorem 1.1]{MM}).

\subsection{Finite unions of interpolating sequences}

Interpolation can be considered also with multiplicities, or more generally, 
with divided differences.
We show next that a discrete sequence $Z=(z_k)_k$ of the unit disk is the union 
of
$n$ interpolating sequences for the Nevanlinna class $\N$ if and only if
the trace $\N|Z$ coincides with the space of functions on
$Z$ for which the pseudohyperbolic divided differences of order $n-1$ are 
uniformly
controlled by a positive harmonic function. 

In Section~\ref{wep} we will see other characterizations of finite unions of 
Nevanlinna interpolating sequences.

\begin{definition}
Let $Z=(z_k)_k$ be a discrete sequence in $\D$ and let $\omega$ be a function 
given on
$Z$. The \emph{pseudohyperbolic divided differences of  $\omega$} are 
defined by induction
as follows 
\[
\begin{split}
\Delta^0 \omega(z_{k_1})  &=\omega(z_{k_1})\ ,\\
\Delta^j\omega(z_{k_1},\ldots,z_{k_{j+1}}) 
&=\displaystyle\frac{\Delta^{j-1}\omega(z_{k_2},\ldots,z_{k_{j+1}})-\Delta^{
j-1}\omega(z_{k_1},\ldots,z_{k_j})}{b_{z_{k_1}}(z_{k_{j+1}})}\qquad 
j\geq
1.\\
\end{split}
\]

For any $n\in \mathbb N$, denote 
\[
Z^n=\{(z_{k_{1}},\ldots,z_{k_{n}})\in
Z\times\stackrel{\stackrel{n}{\smile}}{\cdots}\times Z\; :\; 
k_j\not=k_l\ \textrm{if}\  j\not=l\},
\] 
and consider the set $X^{n-1}(Z)$ consisting of the
functions defined in $Z$ with divided differences of order $n-1$ uniformly
controlled by a positive harmonic function $H$ i.e., such that for some 
$H\in\Har_+(\D)$,
\[
\sup_{(z_{k_{1}},\ldots,z_{k_{n}})\in Z^n} \vert
\Delta^{n-1}\omega(z_{k_{1}},\ldots,z_{k_{n}})\vert
e^{-[H(z_{k_{1}})+\cdots+H(z_{k_{n}})]}<+\infty\ .
\]
\end{definition}

It is not difficult to see that 
$X^n(Z)\subset X^{n-1}(Z)\subset\cdots\subset X^0(Z)=\N(Z)$. For example,
if $\omega\in X^1(Z)$, we can take a fixed $z_{k_0}\in Z$ and write
\[
 \omega(z_k)=\frac{\omega(z_k)-\omega(z_{k_0})}{b_{z_k}(z_{k_0})} 
b_{z_k}(z_{k_0})
 + \omega (z_{k_0}).
\]
Using that there exists $H\in\Har_+(\D)$ with
\[
 |\Delta^1(z_k, z_{k_0})|=\bigl|\frac{\omega(z_k)-
 \omega(z_{k_0})}{b_{z_k}(z_{k_0})}\bigr|\leq e^{H(z_k)+H(z_{k_0})}
\]
we readily see that there is $H_2\in \Har_+(\D)$, depending on $H$ and 
$z_{k_0}$, such that
\[
 |\omega(z_k)|\leq e^{H_2(z_k)}\ \quad k\in\mathbb N.
\]

The following result 
is the analogue of Vasyunin's description of the
sequences $Z$ in $\D$ such that the trace of the algebra 
$\H^\infty$ on $Z$ equals the space of
pseudohyperbolic divided differences of order $n$ (see \cite{Vas83}, 
\cite{Vas84}).
Similar results hold also for Hardy spaces (see \cite{BNO} and \cite{H}) and 
the 
H\"ormander algebras, both in $\C$ and in $\D$ \cite{MOO}.

\begin{theorem}[Main Theorem \cite{HMN2}]
\label{divided differences}
The trace $\N| Z$ of $\N$ on $Z$ coincides with the set
$X^{n-1}(Z)$  if and only if $Z$ is
the union of $n$ interpolating sequences for $\N$.
\end{theorem}

\section{Sampling sets}\label{sampling}

\subsection{Sets of determination for $\hid$.}
One may consider the dual problem to interpolation: which sequences are 
``thick" enough so that the norm of a function can be computed from its values 
on the sequence? 
Here there is no reason to restrict ourselves to sequences.

\begin{definition}
\label{determin}
We say that $\Lambda \subset \D$ is a \emph{set of determination} for 
$\H^\infty $
if for any $f \in \H^\infty $, $\|f\|_\infty = \sup\limits_{z\in \Lambda} 
|f(z)|$.
\end{definition}

Recall that we say that a sequence $(z_k)_k$ converges to $z^* \in \partial \D$ 
\emph{non-tangentially}
if $\lim_{k\to\infty} z_k=z^*$ and
if there exists $A>0$ such that for all $k$, $|z^*-z_k| \le (1+A) (1-|z_k|)$. 
We write $NT\lim_{z\to\eit} f(z)=\lambda$ if the limit is achieved over all 
sequences tending to $\eit$ non-tangentially.

Also, as commented in the Preliminaries, for $f\in\H^\infty$
the non-tangential boundary value $f^*(\eit)=NT\lim_{z\to\eit} f(z)$
exists a.e. $\theta\in[0,2\pi)$ (see for instance 
\cite[Theorem 3.1, p. 557]{Ga}).

Determination sets for $\H^\infty$ are characterized by a simple geometric 
condition.

\begin{theorem}[Brown, Shields and Zeller, 1960 \cite{BrShZe}]
\label{thbsz}
$\Lambda$ is a set of determination  for $\H^\infty$  if and only if
the set $NT(\Lambda)$ consisting of the $\zeta\in\partial\D$ which
are a non-tangential limit of a sequence of points in $\Lambda$
has full measure, i.e. $|NT(\Lambda)|=2\pi$.
\end{theorem}

\subsection{Defining the question for $\nev$.} 
In general a sequence $Z=(z_k)_k$ is called ``sampling'' for a space of
holomorphic functions $X$ when any function $f\in X$ is determined by its
restriction $f|Z$, with control of
norms. For for the Nevanlinna class, which has nothing like a norm, the 
situation is not so obvious.
We start from the notion of set of determination for $\hid$. Instead of 
requiring
that the least upper bound obtained from the values of $f|Z$ be the same as 
$\sup_{\D}|f|$, 
we will consider the set of harmonic majorants of $\left( \log_+ |f|\right)|Z$ 
compared
to the set of harmonic majorants of $ \log_+ |f|$.

Recall that the topology on $\nev$ is defined with the help of the functional 
$N$, 
defined in \eqref{ndef} and \eqref{ndef2}. We give a variant. 
\[
N_+(f)=\lim_{r\to 1}\frac 1{2\pi}\int_0^{2\pi}\log_+|f(r e^{i\theta})|d\theta=
\inf\bigl\{h(0)\, :\, \textrm{$h\in\Har_+(\D)$ with $\log_+|f|\leq h$}\bigr\}\ .
\]
Notice that the expression in \eqref{ndef2} and the second expression in the 
equation above
make sense for any measurable function on the disk, and we will apply them to 
$f|Z$.

\begin{theorem}{\cite[Theorem 2.2]{MT}}
\label{equivalencia}
The following properties of $Z=(z_k)_k\subset\D$ are equivalent:
\begin{itemize}
\item[(a)] There exists $C>0$ such that for any $f\in\nev$, $N(f) \le N(f|Z)+C$.
\item[(b)] For any $f\in\nev$, $N_+(f) = N_+(f|Z)$.
\item[(c)] $Z$ is a  set of determination for $\nev$, i.e. 
any $f\in \nev$ with $\sup_Z |f|<\infty$ must be
bounded on the whole unit disk.
\item[(d)] For any $f\in\nev$ and $h\in\Har_+(\D)$  such that 
$\log_+|f(z_k)|\leq
h(z_k)$ for all $k$, then necessarily $\log_+|f|\leq h$ on the whole unit disk.
\end{itemize}
\end{theorem}
We say that $Z$ is of determination (or  sampling) for $\nev$ if the above 
properties are satisfied.

\subsection{Main result.}

If property (b) above is satisfied, then in particular it must hold for 
zero-free functions in the Nevanlinna
class, and passing to $\log |f|$, the set $Z$ must be a set of determination 
for the class
$\Har_{\pm}(\D)$ of harmonic functions which are 
the difference of two positive harmonic functions. 

To state subsequent results, we need a variant of the decomposition of the disc 
into Whitney squares, and points in them.
Given $n\in\mathbb N$ and $k\in\{0,\dots,2^n-1\}$, let
\begin{align}
\label{whit}
S_{n,k} &:= \left\{ r \eit :1- 2^{-n} \le r \le 1-
2^{-n-1} , \theta \in [{2\pi}k 2^{-n-4},{2\pi}(k+1)
2^{-n-4}] \right\}; 
\\
z_{n,k} &:= (1- 2^{-n})\exp({2\pi}i k 2^{-n-4}). \nonumber
\end{align}

\begin{theorem}{(Hayman-Lyons, \cite{HaLy})}
\label{HayLyons}
Let $Z \subset \D$.
The following properties are equivalent.
\begin{itemize}
\item[(a)] $\sup_{Z} h  = \sup_{ \D} h\ $ for all $h \in \Har_{\pm}(\D)$.
\item[(b)]
For every $\zeta \in \partial \D$, $\sum\limits_{(n,k): Z\cap S_{n,k} \neq 
\emptyset} 2^{-n} 
 P_{z_{n,k}}(\zeta)  = \infty$.
\end{itemize}
\end{theorem}
Note that in contrast with the condition in Theorem \ref{thbsz}, the condition 
of accumulation to the boundary
must be met at every point with no exception. In particular, a set $Z$ such 
that every boundary 
point is a non-tangential limit of points of $Z$ will satisfy the above 
property.

The sets of determination for $\nev$ have been characterized by S. Gardiner 
\cite{Gd}. We need an auxiliary quantity
depending on a set $A\subset \D$ and $t\ge 0$.  If either $A=\emptyset$ or 
$t=0$, we set
$\mathcal Q (A,t)=0$; otherwise,
\[
\mathcal Q (A,t):=\min \Bigl\{ k \in \mathbb N: \exists \xi_1, \dots, \xi_k \in 
\C \mbox{ such that }
\sum_{1\le j \le k} \log \frac1{|z-\xi_j|} \ge t, \quad \forall z \in A
\Bigr\}.
\]
Finally, for any set $A$ and $\lambda>0$, $\lambda A:= \{ \lambda z, z \in A\}$.

\begin{theorem}(Gardiner \cite{Gd})
\label{gardi}
Let $Z\subset \D$. The following conditions are equivalent:
\begin{itemize}
\item[(a)] $Z$ is a set of determination for $\nev$;
\item[(b)] For every $\zeta \in \partial \D$, $\sum_{n,k} 2^{-n} \mathcal Q 
\left( 2^n (Z\cap S_{n,k}), 
\lfloor P_{z_{n,k}}(\zeta) \rfloor \right) = \infty$.
\end{itemize}
\end{theorem}
One can show by an elementary computation \cite{Gd} that 
\[
\mathcal Q \left( 2^n (Z\cap S_{n,k}), \lfloor P_{z_{n,k}}(\zeta) \rfloor 
\right) \le 4  P_{z_{n,k}}(\zeta),
\]
so that condition (b) in Theorem \ref{gardi} implies condition (b) in Theorem 
\ref{HayLyons}, as could be expected.

\subsection{Regular discrete sets.}
In this section we give some computable conditions for regular sets.
\label{finenets}
\begin{definition}
Let $g:(0,1]\longrightarrow(0,1]$ be a non-decreasing continuous function with 
$g(0)=0$. 
A sequence $Z=(z_k)_k$ is called a $g$\emph{-net} if
and only if 
\begin{itemize}
\item[(i)] The disks
$D (z_k,   g(1-|z_k|))$, $k\in \mathbb N$, are mutually disjoint,
\item[(ii)] There exists $C>0$ such that 
$\bigcup_{k\in \mathbb N} D (z_k, C  g(1-|z_k|)) = \D$.
\end{itemize}
\end{definition}
Another way to think of it is that a $g$-net is a maximal $g$-separated 
sequence. 
For $n$ large enough, up to multiplicative constants, a $g$-net will have 
$g(2^{-n})^{-2}$ points
in each domain $S_{n,k}$. 
\begin{theorem}{\cite[Theorem 4.1]{MT}}
\label{thmxarxes}
Let $Z$ be a $g$-net.
The following properties are equivalent:
\begin{itemize}
\item[(a)] $Z$ is a set of determination for $\nev$.
\item[(b)] $\displaystyle \int_0 \frac{dt}{t^{1/2} g(t)} = \infty$.
\item[(c)] $\sum_n 1/(2^{-n/2} g(2^{-n})) = \infty$.
\end{itemize}
\end{theorem}

Taking regular sets with points which have a different distance in the radial 
and angular directions 
does not change things much. We describe a family of examples.
Let $(r_m)_m \subset (0,1)$ be an increasing sequence of radii with $\lim_m 
r_m=1$ and
$\sup_m \frac{1-r_{m+1}}{1-r_m}<1$. Let $\epsilon_m$ be a decreasing sequence
of hyperbolic distances such that $\lim_m \epsilon_m=0$. The 
\emph{discretized rings associated
to $(r_m)_m$ and $(\epsilon_m)_m$} is the sequence 
$\Lambda=(\lambda_{m,j})_{m,j}$,
where
\[
\lambda_{m,j}=r_m \exp\left( j \frac{2\pi i}{(1-r_m)\epsilon_m} \right)\qquad 
m\in\mathbb N,
\quad 0\leq j<\left[\frac{1}{(1-r_m)\epsilon_m}\right]\ .
\]

\begin{theorem}{\cite[Theorem 4.4]{MT}}
\label{thmlines}
Let $\Lambda=(\lambda_{m,j})_{m,j}$ be the sequence of discretized rings 
associated to $(r_m)_m$ and $(\epsilon_m)_m$. Then 
$Z$ is a set of determination for $\nev$ if and only if
\[
\sum_{m=0}^\infty\left(\frac{1-r_m}{\epsilon_m}\right)^{1/2}=\infty.
\]
\end{theorem}

\subsection{Uniformly dense disks}\label{uniformly}

The sampling sets we consider here are no longer discrete. Recall
a class of sequences considered by Ortega-Cerd\`a and Seip in \cite{OrSe}.
\begin{definition}
A sequence $\Lambda=(\lambda_{k})_{k}\subset\D$ is \emph{uniformly dense} if
\begin{itemize}
\item[(i)] $\Lambda$ is separated, i.e. 
$\inf_{j\neq k} \rho(\lambda_j,\lambda_k)>0$.
\item[(ii)] There exists $r<1$ such that $\D=\bigcup_{k\in\mathbb N} 
D(\lambda_k,r)$.
\end{itemize}
\end{definition}
In the terminology of the previous subsection, those are $g$-nets with a 
constant $g$.

Let $\varphi$ be a non-decreasing continuous function, bounded by some
constant less than $1$. Given a uniformly dense $\Lambda$,  define 
$D_k^{\varphi}=D (\lambda_k , \varphi(1-|\lambda_k|))$
and 
\[
\Lambda(\varphi) := \bigcup_{k\in\mathbb N} D_{k}^{\varphi} .
\]

\begin{theorem}{\cite[Theorem 4.5]{MT}}
\label{ud}
The set $\Lambda(\varphi)$ is sampling for $\nev$ if and only if
\begin{equation}
\label{LyubSeip}
\int_0^1 \frac{dt}{t \log (1/\varphi(t))} =
\infty\ .
\end{equation}
\end{theorem}

This condition actually also characterizes
determination sets for the space of subharmonic functions in the disk having 
the 
characteristic growth of the Nevanlinna class \cite[Section 5]{MT}.


Condition \eqref{LyubSeip} is equivalent to the fact that the harmonic
measure of the exterior boundary $\partial\D$ of $\D \setminus 
\overline{\Lambda(\varphi)}$
is zero, see \cite[Theorem 1]{OrSe}. Notice also
that for any fixed $K>1$, condition \eqref{LyubSeip}
is equivalent to
\begin{equation}\label{LBsuma}
\sum_n \frac1{\log (1/\varphi(K^{-n}))}= \infty\ .
\end{equation}


The above family of examples allows
us to see that
there is no general relationship between $\aa$-sampling sets and Nevanlinna 
sampling sets.
We recall this former, and better known, notion of sampling.

\begin{definition}
A set $\Lambda\subset\D$ is \emph{sampling} for the space
\[
A^{-\alpha}=\{f\in\Hol (\D) : \|f\|_{\alpha}:=\sup_{z\in\D} 
(1-|z|)^\alpha|f(z)|<\infty\}\qquad \alpha>0,
\]
when there exists $C>0$ such that $\|f\|_{\alpha}\leq C 
\sup\limits_{\lambda\in\Lambda} (1-|\lambda|)^\alpha|f(z)|$
for all $f\in A^{-\alpha}$.
\end{definition}

A well-known result of K. Seip \cite[Theorem 1.1]{Se} characterizes 
$\aa$-sampling sets as those $\Lambda$ for which there 
exists a separated subsequence $\Lambda'=(\lambda_k)_k\subset\Lambda$ such that
\[
D_-(\Lambda'):=\liminf_{r\to 1^{-}}\inf_{z\in\D}
\frac{\sum\limits_{k: 1/2<\rho(\lambda_k,z)<r}\log\frac 
1{\rho(\lambda_k,z)}}{\log\frac 1{1-r}}>\alpha\ .
\]

Let $\Lambda_g$ be a fine net associated to a function $g$ with 
$\int_0\frac{dt}{t^{1/2} g(t)}<\infty$, for instance $g(t)=t^{1/4}$. 
According to Theorem~\ref{thmxarxes}, $\Lambda_g$ is not a Nevanlinna sampling 
set. On the other hand, for any given $\alpha>0$, 
we can extract a maximal separated sequence $\Lambda'$ with a separation 
constant small enough so that $D_-(\Lambda')>\alpha$, 
hence $\Lambda_g$ is $\aa$-sampling for all $\alpha>0$.

Also, given $\alpha>0$, consider a uniformly dense sequence $\Lambda$ with 
$D_-(\Lambda)<\alpha$
and take $\varphi$ satisfying $\lim_{t\to 0}\varphi(t)=0$ and \eqref{LyubSeip}. 
Then 
according to Theorem \ref{ud}, $\Lambda(\varphi)$ is Nevanlinna sampling; but 
it is not $\aa$-sampling, since 
$D_-(\Lambda')<\alpha$ for any separated $\Lambda'\subset \Lambda(\varphi)$. 

Alternatively, take a set $\Lambda$ as in Theorem \ref{thmlines}, sampling for 
$\nev$, 
with $\lim\limits_{n\to \infty} \frac{1-r_{n+1}}{1-r_n} = 0$. Then 
$D_-(\Lambda)=0$, so it cannot be $\aa$-sampling for any $\alpha>0$.

\section{Finitely Generated Ideals}\label{ideals}

Here we report mostly on \cite{HMN1}.

In the study of the uniform algebra $\H^\infty $ it is important to know its 
maximal ideals, and thus to know whether 
an ideal is or is not the whole of the algebra.  When considering the ideal 
generated by functions
$f_1, \dots, f_n \in \H^\infty $, an easy necessary condition for it to be 
the whole of the algebra is
$\inf_{z\in \D} \left( |f_1(z)| +\cdots + |f_n(z)| \right) >0$. This turns out 
to be sufficient:
this is the content of the Corona Theorem.

\begin{theorem}[Carleson 1962 ; see \cite{Ga} or \cite{NikTr})]
\label{corona}
If $f_1, \dots, f_n \in \H^\infty $ and
\[
\inf_{z\in \D} \left( |f_1(z)| +\cdots + |f_n(z)| \right) >0,
\] then there 
exist 
$g_1, \dots, g_n \in \H^\infty $  such that $f_1 g_1 +\cdots + f_n g_n \equiv 
1$.
\end{theorem}

More generally, 
we denote by $I_{\H^\infty}(f_1,\ldots, f_n)$ the ideal generated by the 
functions $f_1,\ldots,f_n$ in $\H^{\infty}$.
The general structure of these ideals is not well understood (see the 
references 
in \cite{HMN1}).
\cite{Mort1}, 

In certain situations the ideals can be characterized by
growth conditions. In this context, the following ideals have been studied:
\[ 
 J_{\H^\infty}(f_1,\ldots,f_n)=\Bigl\{f\in \H^{\infty}:\, \exists c= c(f)>0\, 
,\, 
|f(z)|\le c\sum_{i=1}^n |f_i(z)|\, ,\, z\in \D\Bigr\}.
 \]
It is obvious that $I_{\H^\infty}(f_1,\ldots,f_n)\subset 
J_{\H^\infty}(f_1,\ldots,f_n)$. This leads us to the type of results we 
are interested in here.
Tolokonnikov \cite{Tol83} proved that the following conditions are equivalent:
\begin{itemize}
\item [(a)] $ J_{\H^\infty}(f_1,\ldots,f_n)$ contains an interpolating Blaschke 
product for $\H^\infty$, 
\item [(b)] $I_{\H^\infty} (f_1,\ldots,f_n)$ contains an interpolating Blaschke 
product for $\H^\infty$, 
\item [(c)] $\inf\limits_{z\in\D} \sum_{i=1}^n (|f_i(z)|+(1-|z|^2)|f_i'(z)|)>0$.
\end{itemize}

As it turns out, in the special situation of two generators with no common zeros
these conditions are equivalent to $ 
I_{\H^\infty}(f_1,f_2)=J_{\H^\infty}(f_1,f_2)$.
In the case of two generators $f_1$ and $f_2$ with common zeros, we have 
$I_{\H^\infty}(f_1,f_2) = J_{\H^\infty}(f_1,f_2)$ if and only if 
$I_{\H^\infty}(f_1,f_2)$ contains a 
function of the form $BB_{1,2}$ where $B$ is a $\H^\infty$-interpolating 
Blaschke product and $B_{1,2}$ 
is the Blaschke product formed with the common zeros of $f_1$ and $f_2$  (see 
\cite{GMN2}).

For the Nevanlinna class R. Mortini observed that a well-known result of T. 
Wolff implies the following corona theorem (see \cite{Mort0} or  \cite{Mart}).

\begin{theorem}(R. Mortini, \cite{Mort0})\label{N-corona}
Let $I(f_1,\ldots,f_n)$ denote the ideal generated in $\N$ by a given family of 
functions 
$f_1,\ldots, f_m\in \N$. Then $I(f_1,\ldots,f_n)=\N$ if and only if there 
exists 
$H\in \Har_+(\D)$
such that
\[
 \sum_{i=1}^n |f_i(z)|\ge e^{-H(z)},\quad z\in \D.
\]
\end{theorem}

The necessary condition is clear since the B\'ezout equation implies that
$1 \leq \max_j( |f_j|) \sum_j |g_j|$.
As observed in \cite{Mort0}, the sufficient condition is a corollary of a 
theorem of Wolff \cite[Theorem 8.2.3, p. 239]{Ga}.
\begin{theorem}
If $g, g_1, \dots, g_N \in H^\infty (\D)$ and for all $z\in \D$,
$ |g_1(z)| +\cdots + |g_N(z)| > |g(z)|$, then there exist
$f_1, \dots, f_N \in H^\infty (\D)$  such that $f_1 g_1 +\cdots + f_Ng_N = g^3$.
\end{theorem}
It will be enough to take a holomorphic $g$ such that $|g|=e^{-H}$, and conclude
using the fact that $g^3$ is invertible in $\N$.

The ideal corresponding to $J_{\H^{\infty}}$ in $\N$ is defined as:
\[ 
 J(f_1,\ldots,f_n)=\Bigl\{f\in N:\exists H=H(f)\in \Har_+(\D)\, ,\, |f(z)|\le 
e^{H(z)}\sum_{i=1}^n |f_i(z)|\, ,\, z\in \D\Bigr\}.
\]
It is clear that $I(f_1,\ldots,f_n)\subset J(f_1,\ldots,f_n)$. Notice also 
that, by the  previous corona theorem, in the case
when $J(f_1,\ldots,f_n)=\N$, then $I(f_1,\ldots,f_n)=\N$. 

The analogues for $\N$ of the results mentioned above in the context of 
$\H^\infty$ 
read as follows. We see, again, that the general principle of substituting 
$\H^\infty$ by $\N$ and boundedness by a control by a positive harmonic function
remains valid.

\begin{theorem}\label{MainThm}
Let $f_1,\ldots,f_n$ be functions in $\N$. Then the following conditions are 
equivalent:
\begin{itemize}
\item  [(a)] $I(f_1,\ldots,f_n)$ contains a Nevanlinna interpolating Blaschke 
product,  
\item  [(b)]$J(f_1,\ldots,f_n)$ contains a Nevanlinna interpolating Blaschke 
product, 
\item  [(c)] There exists a function $H\in \Har_+(\D)$ such that 
\[
\sum_{i=1}^n (|f_i(z)|+(1-|z|^2)|f_i'(z)|)\ge e^{-H(z)}\, , \ z\in \D.
\]
\end{itemize}
In case $n=2$, if $f_1$ and $f_2$ have no common zeros, the above conditions 
are 
equivalent to 

\begin{itemize}
\item [(d)] $I(f_1,f_2)=J(f_1,f_2)$.
\end{itemize}
\end{theorem}

As in $\H^{\infty}$, each of the conditions (a)-(c) implies $I(f_1,\ldots, 
f_n)=J(f_1,\ldots,f_n)$. On the other hand, an example similar to that
given for the $\H^\infty$ case shows that when $n\ge 3$, 
the converse fails. This example is given by two Nevanlinna interpolating 
Blaschke products $B_1$, $B_2$ with respective zero-sets $Z_1$, $Z_2$. 
In this situation
\[
 I(B_1^2, B_2^2, B_1 B_2)=J(B_1^2, B_2^2, B_1 B_2).
\]
But if $Z_1$ and $Z_2$ are too close, then condition (c) in 
Theorem~\ref{MainThm} cannot hold.

Also, like in the $\H^{\infty}$-situation, if the two generators 
$f_1$ and $f_2$ have common zeros, then $I(f_1,f_2)=J(f_1,f_2)$ if and only if 
$I(f_1,f_2)$ contains a function of the form $BB_{1,2}$ where $B$ is a 
Nevanlinna 
interpolating Blaschke product and $B_{1,2}$ is the Blaschke product formed 
with the 
common zeros of $f_1$ and $f_2$. 

The proof Theorem \ref{MainThm}
uses some of the ideas from the $\H^\infty$ case, but also some 
specific properties of the Nevanlinna class, in particular the 
description of Nevanlinna interpolating sequences in terms of harmonic measure 
seen in 
Theorem\ref{alt-interpolation}(d).

\subsection{Stable rank and two open problems}

The ideal generated by
an $n$-tuple of functions $f_1, \dots, f_n \in \H^\infty $ satisfying the 
conclusion of the Corona Theorem \ref{corona}
is the whole of $\hid$.  In general, in an algebra $A$ (or even a unitary ring) 
an $n$-tuple 
which is not contained in any ideal smaller than $A$ is called 
\emph{unimodular}.  
An unimodular $n+1$-tuple $f_1, \dots, f_n, f_{n+1} \in A$ such that there 
exist $h_1, \dots, h_n \in A$ 
with $f_1+h_1 f_{n+1}, \dots, f_n+h_n f_{n+1}$ unimodular is called 
\emph{reducible}.

\begin{definition}
The \emph{stable rank} of an algebra $A$ is the smallest integer $n$ such that 
any $n+1$-tuple in $A$ is reducible.
\end{definition}

It is known, for example, that the stable rank of the disc algebra 
(consisting of the holomorphic functions countinuous up to the boundary)
has stable rank 1 (see \cite{CS}, \cite{JMW}). The following result by Treil 
goes in the same direction.

\begin{theorem}[Treil, 1992, \cite{Tr}]
The stable rank of $\hid$ is equal to $1$; more explicitly, given any $f_1, 
f_2\in \hid$ such that 
$\inf_{z\in\D} \left( |f_1(z)| + |f_2(z)| \right) >0$, there exists $h\in\hid$ 
such that 
$f_1+h f_2$  is invertible in $\hid$.
\end{theorem}

As far as we know, the stable rank for the Nevanlinna class
is unknown, but the following result shows that it is at least two.

\begin{proposition} (\cite[Proposition 5.1]{HMN1})\label{rkge2}
The stable rank of the Nevanlinna class is at least 2.
\end{proposition}

{\bf Open problem:} What is the stable rank of $\N$?

Our guess is that it is either two or infinity. In support of the first option
is that any triple $(f_1 , f_2 , f_3) \in \N^3$ such that for 
some $i$ the zeros of $f_i$ form a Nevanlinna interpolating sequence, can be 
reduced. 

In order to see that the stable rank of $\N$ cannot be 1 we construct
a pair of Blaschke products $B_1$, $B_2$ such that $(B_1,B_2)$ is unimodular 
but for which there are no $\phi\in\N$ and no $H$ harmonic in $\D$ such that
\begin{equation}\label{sr}
 \log|B_1(z)+\phi(z) B_2(z)|= H(z),\quad z\in\D.
\end{equation}
This prevents the existence of $e^f\in\N$, invertible in $\N$, with
\[
 B_1+\phi B_2=e^f.
\]

Let $Z_1=(z_k)_k:=(1-2^{-k})_k$ and $B_1$ its 
associated Blaschke product. Take now
a point $\mu_k\in (0,1)$ close enough to $z_k$ so that
\begin{equation*}\label{IC}
 |B_1(\mu_k)|=\left\{
 \begin{array}{ll}
  e^{-\frac{1}{1-|z_k|^2}} &\text{if $k$ even}\\
  e^{-\frac{2}{1-|z_k|^2}} &\text{if $k$ odd}.
 \end{array}
 \right.
\end{equation*}
Set $Z_2=(\mu_k)_k$ and $B_2$ its Blaschke product. 

Since both $Z_1$ and $Z_2$ are $\H^\infty$-interpolating and
on the segment $(0,1)$, is not difficult to see that for some $c>0$
\[
 |B_1(z)|+|B_2(z)|\geq e^{-\Re\bigl(c\frac{1+z}{1-z}\bigr)}, \quad z\in\D,
\]
hence $(B_1, B_2)$ is unimodular in $\N$.

On the other hand, for any $\phi\in\N$,
\[
 (1-|\mu_k|^2) \log|B_1(\mu_k)+\phi(\mu_k) B_2(\mu_k)|=
 (1-|\mu_k|^2)\log|B_1(\mu_k)|=
 \begin{cases}
  -1 & \textrm{if $k$ even}\\
  -2 & \textrm{if $k$ odd}.
 \end{cases}
\]
This prevents \eqref{sr} to hold for any $H=P[\nu]$ harmonic,
with $\nu$ finite measure on $\partial\D$, since
\[
 \lim_{k\to\infty} (1-|\mu_k|^2) P[\nu](\mu_k)=\nu (\{1\}).
\]

\subsection{The $f^2$ problem}

In the late seventies T. Wolff presented a problem on ideals of $\H^\infty$, 
known now as the $f^2$ problem,  which was finally solved by S. Treil in 
\cite{Tr2}. The analogue or Nevanlinna class is the following: let 
$f_1,\ldots,f_n$ be  functions in the Nevanlinna class, and let $f\in \N$ be 
such 
that there exists $H\in \Har_+(\D)$ with
\begin{equation}\label{cond}
 |f(z)|\le e^{H(z)}(|f_1(z)|+\cdots |f_n(z)|)^p, \quad z\in \D,
\end{equation}
for some $p\ge 1$. Does it follow that $f\in I(f_1,\ldots,f_n)$ ?

As in the $\H^{\infty}$ case, when $p>2$, the $\overline{\partial}$ estimates 
by 
T. Wolff show
that the answer is affirmative. When $p<2$ the answer is in general negative, 
as 
the following example shows.
Let $N$ be an integer such that $N+1>2Np$, $f=B_1^NB_2^N$, $f_1=B_1^{N+1}$ and
$f_2=B_2^{N+1}$. Then \eqref{cond} holds but $f\notin I(f_1,f_2)$ if 
$(B_1,B_2)$ 
is not
unimodular in $\N$.

{\bf Open problem:} What happens in the case  $p=2$?

\section{Weak embedding property and invertibility threshold}\label{wep}

\subsection{Weak embedding property in $\H^\infty $.}

The quotient algebras of $\H^\infty$ arise naturally in several questions, 
notably 
problems of invertibility of operators. The paper \cite{GMN} gives more detail 
about this,
in the more general framework of uniform algebras;
we shall concentrate on the question of invertibility within a quotient algebra 
of $\H^\infty $.

Recall that 
a function $I\in \H^\infty$ is called \emph{inner} if $\lim_{r\to1} 
|I(r\xi)|=1$ for almost every $\xi\in\partial\mathbb{D}$. 
It is a consequence of  Beurling's theorem about cyclic functions in the Hardy 
classes 
that any closed principal ideal of $\H^\infty$ is of the form $I \H^\infty$.
algebra $\H^\infty /I\H^\infty$ and its \emph{visible spectrum,} which consists 
of $\Lambda=Z(I)$, the zeros of $I$ in the open unit disk.  
Given two functions $f, g$ in the same equivalence class $[f] \in \H^\infty 
/I\H^\infty$, they always coincide
on $I^{-1}\{0\}=:Z$. If $[f]$ is invertible in $\H^\infty /I\H^\infty$, then 
$\inf_Z |f| >0$.
One may ask whether the converse holds. This is the $n=1$ case of the following 
property (see \cite{GMN2}).

\begin{definition}
Let $I$ be an inner function with zero set $Z$. 
We say  that $\H^\infty /I\H^\infty$ has the Corona Property if: 
given $f_1,\ldots, f_n\in \H^\infty$ such that 
\begin{equation*}
 |f_1(z)|+\cdots+|f_n(z)|\geq \delta \qquad z\in Z,
\end{equation*}
for some $\delta>0$, 
then there exist $g_1,\ldots, g_n, h\in \H^\infty$ such that 
$f_1g_1+\cdots+f_ng_n=1 +h I$.
\end{definition}

\begin{theorem}[Gorkin, Mortini, Nikolski, \cite{GMN2}]\label{bounded-wep}
Given an inner function $I$ with $Z=I^{-1}\{0\}$,
the following are equivalent:
\begin{itemize}
\item [(a)]
Any $[f] \in \H^\infty /I\H^\infty$ such that $\inf_Z |f| >0$ is invertible;
\item [(b)]
$\H^\infty /I\H^\infty$ has the Corona Property;
\item [(c)]
$I$ satisfies the \emph{Weak Embedding Property} (WEP): 
\begin{equation*}
\mbox{For any }\epsilon>0,\mbox{ there exists }\eta>0\mbox{ such that 
}|I(w)|\geq \eta 
\mbox{ for }  w\mbox{ with }\inf_{z\in Z}\rho(w,z)>\epsilon.
\end{equation*}
\end{itemize}
\end{theorem}
 
Among inner functions with the same zero set, Blaschke products are the 
largest, and thus
most likely to enjoy the WEP. Vasyunin's condition \eqref{vasyu} (see 
Section~\ref{int-equivalent}) 
shows that $\H^\infty$-interpolating
Blaschke products satisfy the weak embedding property
with $\eta= \epsilon/C$. 

Thus for any Carleson-Newman Blaschke product, i.e.  a finite product of 
interpolating Blaschke
products, we will have the WEP with $\eta= \epsilon^N/C$, and if it is 
satisfied for $I$ with that value 
of $\eta$, then $I$ is a  Carleson-Newman Blaschke product \cite{GM}, 
\cite{Bo}. But there are other
functions with the WEP \cite{GMN2}. No geometric characterization is known for 
the WEP, but further examples and
results can be found in \cite{Bo}, \cite{BNT}.

\subsection{Weak embedding property in the Nevanlinna Class.}
In the algebra $\nev$,  any nonvanishing function is invertible and so any 
principal ideal is generated by some
Blaschke product $B$ with zero set $Z$. The elements of the quotient algebra 
$\N_B = \nev / B \nev$
are in one-to-one correspondence with their  traces over $Z$. 

\begin{definition}
Let $B$ be a Blaschke product with zero set $Z$. 
We say that the Corona Property holds for $\N_B$ if for any positive integer 
$n$ and any  $f_1,\ldots,f_n\in \nev$ 
for which there exists $H\in \Har_+(\D)$ such that 
\begin{equation}
\label{ncnev}
|f_1(z)|+\cdots+|f_n(z)|\geq e^{-H(z)}\qquad z\in Z,
\end{equation}
there exist $g_1,\ldots,g_n , h \in \nev$ such that $f_1g_1+\cdots+f_ng_n  = 1 
+  B h $, that is, there exist $g_1,\ldots,g_n \in \nev$ such that 
\[
 f_1(z)g_1(z)+\cdots+f_n(z)g_n(z)=1, \quad z\in Z. 
\]
\end{definition}

Observe that condition \eqref{ncnev} is necessary, and that the case $n=1$ 
simply expresses invertibility in $\N_B$. Observe also that when $Z$ is an 
interpolating sequence and $n=1$,
condition \eqref{ncnev} means that  $\log \left(1/|f(z_k)|\right) \leq H(z_k)$, 
and so the values
$1/f(z_k)$ can be interpolated by $g\in \nev$. So in the Nevanlinna case we  
immediately see that interpolating sequences
are related to the Corona property for the quotient algebra.

As in the case of $\hid$, we can see that the Corona property for $\N_B$ can be 
reduced to bounding $|B|$
from below, except that in this case both the distance condition and the bounds 
have to be expressed
in terms of positive harmonic functions rather than constants.

\begin{theorem}{\cite{MNT}.}
\label{corona_bound}
Let $B$ be a Blaschke product and let $Z$ be its zero sequence. The following 
conditions are equivalent:
\begin{itemize}
\item [(a)] The Corona Property holds for $\N_B$.
\item [(b)] For any $H_1\in  \Har_+(\D)$, there exists $H_2\in  \Har_+(\D)$ 
such that 
$|B(z)|\geq e^{-H_2(z)}$ for any $z\in\mathbb{D}$ such that $\rho(z,Z)\geq 
e^{-H_1(z)}$.
\end{itemize}
\end{theorem}
Blaschke products satisfying those equivalent properties are said to have the 
\emph{Nevanlinna WEP}.
Once this is known, it becomes clear that a finite product of
Nevanlinna-interpolating Blaschke products will have the Nevanlinna WEP. Assume 
that
 $Z=\cup_{j=1}^N Z_j$, with $Z_j$ Nevanlinna interpolating, and let 
$H_j\in\Har_+(\D)$ be such that
 the estimate in Theorem \ref{alt-interpolation}(b) holds for $Z_j$ and its 
associated Blaschke product $B_j$. 
 Now, given $z\in\D$ with
 $\rho(z,Z) \ge e^{-H_0(z)}$,
\[
|B(z)| = \prod_{j=1}^N |B_j(z)| \ge \prod_{j=1}^N e^{-H_j(z)}\rho(z,Z_j)
\ge \exp{\Bigl(-\bigl( \sum_{j=1}^N  H_j(z)\bigr)-NH_0(z)\Bigr)},
\]
so Theorem \ref{alt-interpolation}(b) holds for the whole $Z$.

In contrast with the $\hid$ case, the converse to this holds. 
\begin{theorem}{\cite{MNT}.}
\label{finite_u}
The Corona Property holds for $\N_B$ if and only if $B$ is a finite product of 
Nevanlinna interpolating Blaschke products.
\end{theorem}

\subsection{Some elements of proof.}

The class of  Nevanlinna WEP sequences is thus quite different, and much 
simpler, than that of WEP sequences.
For instance, in the setting of $\hid$, any WEP Blaschke product which is not a 
finite union of 
interpolating Blaschke products admits a subproduct which fails to be WEP 
\cite[Lemma 1.3]{MNT}.
The following result follows immediately from Theorem \ref{finite_u}, but is 
actually 
a step in its proof.
\begin{lemma}
\label{subproduct}
 Any subproduct of a Nevanlinna WEP Blaschke product is also a Nevanlinna WEP 
Blaschke product.
\end{lemma}
To give a flavor of how the Nevanlinna version of the WEP differs from the 
$\hid$ version, we will sketch
the proof of this. First we observe that points which are ``far away'' from the 
zero set, in the sense of distances
measured with positive harmonic functions, can be found in a neighborhood of 
any point.

We will need an auxiliary function. 
\begin{definition}
\label{harmbase}
Given a Blaschke sequence $Z=(z_k)_k$, let $H_Z$ denote the positive harmonic 
function defined by
\begin{equation}
\label{funcio}
H_Z(z)=\sum_k\int_{I_k}\frac{1-|z|^2}{|\xi-z|^2}|\mathrm{d}\,\xi|,\hspace{0.5cm}
z\in\mathbb{D},
\end{equation}
where $I_k:=\{\xi\in\partial\mathbb{D}:|\xi-z_k / |z_k||\leq 1-|z_k|\}$ denotes 
the Privalov 
shadow of $z_k$.
\end{definition}
\begin{lemma}{\cite[Lemma 1.1]{MNT}.}
\label{far}
 There exists $c_0>0$ such that for all $H\in\Har_+(\D)$ such that $H\geq c_0 
H_Z$ the following property holds: 
 for all $z\in\D$ there exists $ \tilde z$ such that 
 $\rho(\tilde z, z)\leq e^{- H(z)}$ and
 $\rho(\tilde z, Z)\geq e^{-10 H(z)}$.
\end{lemma}
This is achieved through a simple counting argument. Note that $e^{-10 H(z)} << 
e^{- H(z)}$, although the respective
harmonic functions differ only by a constant.
\begin{proof*}{\it Proof of Lemma \ref{subproduct}.} 
 Assume $B=B_1B_2$ is a Nevanlinna WEP Blaschke product. Denote  by $Z_i$ the 
zero sequences of $B_i$, $i=1,2$. 
 Let 
$z\in \D$, and
 $H_1\in\Har_+(\D)$ be such that $\rho(z,Z_1)\geq e^{-H_1(z)}$. If needed,
 make $H_1$ larger so that  Lemma~\ref{far} applies. 
If $z$ verifies $\rho(z,\Lambda_2)\geq e^{-10 H_1(z)}$
as well,
 since $B$ is Nevanlinna WEP, 
 there exists $H_2\in\Har_+(\D)$ such that  
 \[
  |B_1(z)|\geq |B(z)|\geq e^{-H_2(z)}.
 \]
If on the other hand $\rho(z,\Lambda_2)\leq e^{-10 H_1(z)}$, 
by Lemma~\ref{far} we can pick $\tilde z\in\D$ with $\rho(\tilde z,z)\leq 
e^{-10 H_1(z)}$ 
and $\rho(\tilde z,Z_2)\geq e^{-100 H_1(z)}$. Hence $\rho(\tilde z,Z)\geq 
e^{-100 H_1(z)}$,
so there exists $H_3\in\Har_+(\D)$ such that
 \[
  |B_1(\tilde z)|\geq |B(\tilde z)|\geq e^{-H_3(\tilde z)}.
 \]
 Since $B_1$ has no zeros in $D(z,e^{-5 H_1(z)})$, Harnack's inequalities 
applied in that disc give
$ |B_1( z)|\geq e^{-2 H_3(z)}$.
\end{proof*}

In the $\hid$ framework, one easily sees from Carleson's Theorem \ref{carlint} 
that any Blaschke product with separated
zero set enjoying the WEP is actually an interpolating Blaschke product.  The 
Nevanlinna analogue holds.
\begin{lemma}
\label{sepint}
A Nevanlinna WEP Blaschke product with  weakly separated zero set (as in 
Definition \ref{weaksep})
is a Nevanlinna interpolating Blaschke product. 
\end{lemma}
\begin{proof}
Let $H_1 \in \Har_+(\D)$ giving the separation.
Since $B$ is Nevanlinna WEP there exists $H_2\in\Har_+(\D)$ such that
\[
 |B(z)|\geq e^{-H_2(z)} \quad\textrm{for}\quad z\in\cup_{\lambda\in Z} \partial 
D(\lambda, e^{-H_1(\lambda)})
\]
In particular,
\[
 |B_\lambda(z)|\geq e^{-H_2(z)} \quad\textrm{for}\quad z\in  \partial 
D(\lambda, e^{-H_1(\lambda)}).
\]
Since $B_\lambda$ has no zeros in $D(\lambda, e^{-H_1(\lambda)})$ we can apply 
the maximum principle to the harmonic function 
$\log |B_\lambda|^{-1}$ to deduce that
\[
 |B_\lambda(\lambda)|\geq \min_{z\in  \partial D(\lambda, e^{-H_1(\lambda)})} 
e^{-H_2(z)} \geq e^{-2H_2(\lambda)}.
\]
\end{proof}

The bulk of the proof of Theorem \ref{finite_u} is spent on showing that any 
Nevanlinna WEP Blaschke product
must have a zero set which is a finite union of weakly separated subsequences 
\cite[Theorem B and Section 3]{MNT}. 
Then Lemma \ref{subproduct}
shows that each of those subsequences generates  a Nevanlinna WEP Blaschke 
product, which then must
be Nevanlinna interpolating by Lemma \ref{sepint}.

Our last result collects several different descriptions of  products of exactly 
$N$ Nevanlinna interpolating Blaschke products,
which are then equivalent to the property about the trace space on the zero set
given in Theorem \ref{divided differences}. 
Analogous results for interpolating Blaschke products were proved by 
Kerr-Lawson 
\cite{K-L}, Gorkin and Mortini \cite{GM} and Borichev \cite{Bo}.

Given a Blaschke product $B$ and $z \in \D$, let $|B(N)(z)|$ denote the value 
at 
the point $z\in\mathbb{D}$ of the modulus of the Blaschke product obtained from 
$B$ after deleting the $N$ zeros of $B$  closest to $z$ (in the 
pseudo-hyperbolic metric). 

\begin{theorem}{\cite[Theorem C]{MNT}}
\label{descriptions}
Let $B$ be a Blaschke product with zero set $Z$ and let $N$ be a positive 
integer. 
The following conditions are equivalent:
\begin{itemize}
\item  [(a)] $B$ is a product of $N$ Nevanlinna interpolating Blaschke products.
\item  [(b)]  There exists $H_1\in \Har_+(\D)$ such that 
\begin{equation*}
|B(z)|\geq e^{-H_1(z)}\rho^N\left(z,Z\right),\hspace{1cm}z\in\mathbb{D}.
\end{equation*}
\item  [(c)]  
There exists $H_2\in \Har_+(\D)$ such that $|B(N)(z)|\geq e^{-H_2(z)}$, 
$z\in\mathbb{D}$.
\item [(d)] There exists $H_3\in \Har_+(\D)$ such that
\begin{equation*}
D_N(B)(z)=\sum_{j=0}^{N}(1-|z|)^j|B^{(j)}(z)|\geq 
e^{-H_3(z)},\hspace{1cm}z\in\mathbb{D}.
\end{equation*}
\end{itemize}
\end{theorem}

The equivalence between (a), (b) and (d) for $N=1$ is stated in 
Theorem~\ref{alt-interpolation} 
(see \cite[Theorem 1.2]{HMN1}). 

A consequence of these characterizations is a kind of stability of the property 
of being a 
finite product of Nevanlinna Interpolating Blaschke products.
\begin{corollary}
Let $B$ be a finite product of Nevanlinna interpolating Blaschke products. 
Then, there exists $H_0=H_0(B)\in \Har_+(\D)$  such that for any $g\in 
H^\infty$ 
with $|g(z)|\leq e^{-H_0(z)}$, $z\in\mathbb{D}$, the function $B-g$ factors as 
$B-g=B_1G$, 
where $B_1$ is a finite product of Nevanlinna interpolating Blaschke products 
and 
$G\in H^\infty$ is such that $1/G\in H^\infty$.
\end{corollary}

\subsection{Invertibility threshold in the Nevanlinna Class.}

Along with the definition of the WEP, the even more subtle question of the 
invertibility threshold was raised in \cite{GMN2}: 
is there $c\in [0,1)$ such that given any
$[f] \in \H^\infty /I\H^\infty$ with $\|[f]\|:=\inf\left\{ \|g\|: 
g\in[f]\right\}=1$
and $\inf_{I^{-1}\{0\}} |f| > c$, then $f$ is invertible in $\H^\infty 
/I\H^\infty$? 
The case of the WEP corresponds to $c=0$.
It was shown in \cite{NV} that for any value of $c\in(0,1)$, there is some 
Blaschke product with zero set $Z$ so that under the condition $\inf_Z |f| > 
c$, then $f$ is invertible in $\H^\infty /B\H^\infty$,
but no $c'<c$ will work.  

In the case of the Nevanlinna Class, $|f|$ being bounded from below by a small 
constant 
will be replaced by $|f|$ being bounded from below by $e^{-H}$, where $H$ is a 
large positive harmonic function. 
To study the question, we need a quantitative version of Theorem 
\ref{corona_bound}. 
Since any $f\in\nev$ can be written $f=g_1/g_2$ with $g_1,g_2 \in \hid$, 
invertibility of $f$ is equivalent
to that of $g_1$. We thus state the result for bounded functions, which is a 
way of normalizing the
functions we are considering. 

Observe that a restriction arises which does not occur in the $\hid$ case: the 
two properties below
are equivalent only when considering positive harmonic functions which are 
larger than $H_Z$,
the function in Definition \ref{harmbase}.

\begin{theorem}{\cite[Theorem 1]{NT}}
\label{invert_bound}
Let $B$ be a Blaschke product with zero set  $Z=(z_k)_k$.
    \begin{enumerate}
        \item [(a)]  There exists a universal constant $C>0$ such that the 
following statement holds. 
        Let $H\in \Har_+(\D)$ and assume that the function $-\log|B|$ has a 
harmonic majorant on the set 
        $ \{z\in\D :\rho(z,Z)\geq e^{-H(z)} \}$. Then for any $f\in \H^\infty$, 
$||f||_\infty\leq 1$ such that
            \begin{equation}
\label{parta}
             |f(z_k)|> e^{-CH(z_k)}, \quad k=1,2, \ldots , 
            \end{equation}
        there exist $g,h\in\nev$ such that $fg=1+Bh$.
        \item [(b)] There exist universal constants  $C_0>0$ and $C >0$ such 
that the following statement holds. 
        Let $H\in\Har_+(\D)$ with $H\geq C_0 H_Z$. Assume that for any $f\in 
H^\infty$, $||f||_\infty\leq1$ such that
        \eqref{parta} holds,
        there exist $g,h\in\nev$ such that $fg=1+Bh$. Then, the function 
$-\log|B|$ has a 
        harmonic majorant on the set $\{z\in\mathbb{D}\, : \rho(z,Z)\geq 
e^{-H(z)}\}$.
    \end{enumerate}
\end{theorem}
The result can be extended to B\'ezout equations with any number of generators 
\cite[Corollary 2]{NT}.

So we get a sufficient condition (a) for a solution to the invertibility 
problem in the quotient algebra 
 $\mathcal{N } / B \mathcal{N}$, and a necessary condition when the harmonic 
function under consideration
 is large enough. 
This condition is in terms of the following class.
 
\begin{definition}
Given a Blaschke product $B$, let $\mathcal{H}(B)$ be the set of functions 
$H\in \Har_+(\mathbb{D})$ such that $-\log|B|$ has a harmonic majorant on the 
set $\{z\in\mathbb{D}\,:\rho(z,Z)\geq e^{-H(z)}\}$. 
 \end{definition}
 
  It is easy to see that constant functions are always in $\mathcal{H}(B)$ (see 
Proposition 4.1 of \cite{HMNT}), 
  and that if $H_1\in \mathcal{H}(B)$ and $H_2 \in \mbox{Har}_+(\mathbb{D})$,
 $H_2 \le H_1$, then $H_2\in \mathcal{H}(B)$. In this language Theorem 
\ref{finite_u} reads as follows: 
 $\mathcal{H}(B)=\Har_+(\mathbb{D})$ if and only if $Z$ is a finite union of 
interpolating sequences for $\mathcal{N}$.

For any Blaschke product $B$, $\mathcal{H}(B)$ does contain unbounded functions 
\cite[Theorem 2]{NT}. We have 
the following analogue of the Nikolski-Vasyunin result, showing that given any 
positive
harmonic function $H$, there exist Blaschke products so that $H$ is in some 
sense on the boundary
of the class $\mathcal{H}(B)$. 

\begin{theorem}{\cite[Theorem 3]{NT}}
\label{crith}
\begin{enumerate}
\item [(a)]
Let $H_1,H_2\in \Har_+(\D)$ such that
\begin{equation*}
    \limsup_{|z|\to 1}\;\frac{H_1(z)}{H_2(z)}=+\infty.
\end{equation*}
Then there exists a Blaschke product $B$ with zero set $Z$ such that 
$H_2 \in \mathcal H(B)$ but $H_1 \notin \mathcal H(B)$.
\item [(b)]
For any $\eta_0>0$, and any unbounded positive harmonic function $H$, there 
exists a Blaschke product $B$ such that 
$H \in \mathcal H(B)$ but $(1+\eta_0)H \notin \mathcal H(B)$.
\end{enumerate}
\end{theorem}

The Blaschke products which we exhibit in Theorem \ref{crith} have a 
distribution of zeroes taylored to 
the variation of the positive harmonic functions involved. When the harmonic 
functions we consider are 
``too big'' compared to the density of zeroes of $B$, the delicate phenomenon 
involved in Theorem \ref{crith}(b)
disappears, and multiplying by a constant does not affect membership in 
$\mathcal H(B)$.

Recall that the disc is partitioned in the Whitney squares $S_{n,k}$ defined in 
\eqref{whit}.

\begin{theorem}{\cite[Theorem 4]{NT}}
\label{largeH}
Let $B$ be a Blaschke product with zero set $Z$. Let $H\in\Har_+(\D)$ such that 
\begin{equation}
\label{control}
\inf\{e^{H(z)}:z\in Q\}\geq \#(Z\cap Q),
\end{equation}
for any dyadic Whitney square $Q$. Assume that $H\in \mathcal{H}(B)$; then, for 
any $C>0$, $CH\in \mathcal{H}(B)$.
\end{theorem}
Of course, the result is non-trivial for $C>1$ only. Note that the harmonic 
functions that verify 
the hypothesis of Theorem \ref{invert_bound}(b) will verify \eqref{control}.

It would be nice, given any Blaschke product, to determine the class $\mathcal 
H(B)$. Although this seems out of
reach in general, a sufficient condition is known. Given a dyadic Whitney 
square $Q$, let $z(Q)$ denote its center. 

\begin{theorem}
\label{suffhb}
Let $B$ be a Blaschke product with zero set $Z$. Let $\mathcal{A}$ be the 
collection of dyadic Whitney squares 
$Q$ such that $\#(Z\cap Q)>0$. Let $H\in\Har_+(\D)$. Assume that the map
$ z_Q \mapsto \#(Z\cap Q) \cdot H(z(Q))$ for any $Q\in\mathcal{A}$ (and $0$ 
elsewhere) admits a harmonic majorant. 
Then $H \in \mathcal H(B)$.
\end{theorem}
Notice that we impose no direct restriction on the values of $H$ in the dyadic 
squares where no zero of $B$ 
is present. Moreover, there is a class of Blaschke products for which this 
sufficient condition is also necessary 
\cite[Section 2]{NT}.

\section{The Smirnov class}\label{smirnov}

The usual definition of the Smirnov class is 
\[
 \N_+=\bigl\{ f\in\N : \lim_{r\nearrow 1}\frac 1{2\pi}\int_0^{2\pi}\log_+|f(r 
\eit)|\, d\theta=
 \frac 1{2\pi} \int_0^{2\pi}\log_+|f^*(r\eit)|\, d\theta\bigr\} .
\]

However, as everywhere else, we take the equivalent definition
in terms of harmonic majorants.

\begin{definition}
 A harmonic function $h$ is called \emph{quasi-bounded} if it is the Poisson 
integral of a measure absolutely
continuous with respect to the Lebesgue measure on the circle, i. e. $h=P[w]$ 
for some $w\in L^1(\partial\D)$.
Let $QB(\D)$ denote the set of quasi-bounded harmonic functions, 
and $QB_+(\D)$ the cone of those which are nonnegative.
\end{definition}

\begin{definition}
The \emph{Smirnov class} $\mathcal N_+ $ is the set of $f\in\nev$ such that 
$\log|f|$ 
has a quasi bounded harmonic majorant.
Equivalently, it is the set of $f\in\nev$ such that
$\log|f(z)|\le \mathcal P(\log|f^*|)(z)$.
\end{definition}

In terms of the factorization \eqref{factor}, Smirnov functions are the 
Nevanlinna 
functions with singular factor $S_2\equiv 1$. Then, any $f\in\N_+$ has a 
factorization 
of the form
\[
 f=\alpha\frac {B S_1 \mathcal O_1}{\mathcal O_2},
\]
where $\mathcal O_1, \mathcal O_2$ are outer with $\|\mathcal O_1\|_\infty, 
\|\mathcal O_2\|_\infty\leq 1$,
$S_1$ is singular inner, $B$ is a Blaschke product and $|\alpha|=1$.

In particular, unlike Nevanlinna functions, the inverse of a nonvanishing 
Smirnov function is not necessarily a Smirnov function. Actually, 
the class $\N^+$ is the subalgebra of $\N$
where the invertible functions are exactly the outer functions. 

Other aspects of $\N_+$ are better than in $\N$. For instance, $(\N_+, d)$ is a 
topological vector space (see \cite{SS}). 
It can also be represented as the union of certain weighted Hardy spaces 
$\H^2(w)$, and hence given the inductive limit 
topology \cite{McC}. 

According to the above definition, we expect the results explained in the 
previous sections to hold as well for $\N_+$
as soon as $H\in\Har_+(\D)$ is replaced by $H\in QB_+(\D)$. This is indeed the 
case most of the times, but not always. 

Let us briefly comment how the previous results are modified or adapted to 
$\N_+$.

\subsection{Interpolation in $\N_+$}\label{S-interpolation}
Theorem~\ref{nevinter} holds as expected, replacing $\N$ by $\N_+$ and 
$h\in\Har_+(\D)$ is replaced by $h\in QB_+(\D)$.
As for Corollary~\ref{interpolation-geometric}(a), condition \eqref{blabdd} has 
to be replaced by
 \[
  \lim_{k\to\infty} (1-|z_k|) \log | B_{k} (z_k)|^{-1}=0.
 \]
This is in accordance with the fact that for $h\in QB(\D)$
\[
 NT \lim_{z\to\eit}(1-|z|) h(z)=0,\quad \textrm{for all $\theta\in [0,2\pi)$} .
\]

As for (b), when $\mu_Z=\sum_k (1-|z_k|) \delta_{z_k}$ has bounded Poisson 
balayage,
the same argument as before shows that 
$Z$ is Smirnov interpolating if and only if \eqref{bal} holds. In particular, 
in this situation Nevanlinna and Smirnov interpolating sequences are the same.

Theorem~\ref{alt-interpolation} and Corollary~\ref{stability} work according to 
the general substitution explained above.
Same thing with Theorem~\ref{divided differences}.

As in the Nevanlinna case, the existence of peak functions with uniform bounds 
is not sufficient for interpolation in the Smirnov class. But the situation in this setting
is a bit more subtle. It is known that if $f\in \N_+$
there exists a convex, increasing 
function $\psi:[0,+\infty)\longrightarrow[0,+\infty)$ (depending on $f$), with 
$\lim\limits_{t\to +\infty} \psi(t)/t=+\infty$, and such that
\[
N_\psi(f):=\int_{0}^{2\pi}
 \psi\left[  \log(1+|f^*(e^{i\theta})| \right)]\;  \frac{d\theta}{2\pi}<+\infty\ 
.
\]
Accordingly, when $Z=(z_k)_k$ is interpolating for the Smirnov class there exist $\psi$ as above, 
$C>0$ and functions $f_k\in \N_+$ peaking at $z_k$ (i.e, $f_k$ for which \eqref{1i0} holds) and with
$N_\psi(f_k)\leq C$.

The converse fails: there exist sequences $Z$ which are not Smirnov interpolating but  
for which there exist $\psi$, $C$ and functions $f_k\in\N_+$ peaking at $z_k$ and with
$N_\psi(f_k)\leq C$ (see \cite[Theorem 1.2]{MM}).

\subsection{Sampling sets for $\N_+$}\label{S-sampling}
Sampling (or determination) sets for the Smirnov class can be defined as in 
Section \ref{sampling}, 
with the standard replacement of $\N$ by $\N_+$ and $\Har_+(\D)$ by $QB_+(\D)$. 
Then the analogue of 
Theorem~\ref{equivalencia} holds, but much more can be said.

By Theorem~\ref{thbsz}, for $\Lambda\subset\D$ to be sampling for $\N_+$ it is 
necessary that
$|NT(\Lambda)|=2\pi$. But Smirnov functions are determined by their boundary 
values, so 
this condition is also sufficient (see \cite[Theorem 2.3]{MT}).

\subsection{Finitely generated ideals}\label{S-ideals}

Mortini's solution to the corona problem given in Theorem~\ref{N-corona} works 
also for $\N_+$, 
with the usual substitution explained above. The same thing happens with 
Theorem~\ref{MainThm}.

The situation changes for the stable rank. Notice that the example showing that 
the stable rank of 
$\N$ has to be at least two does not work for $\N_+$, because the function
$\Re(\frac{1+z}{1-z})$ is not quasi-bounded (it is the inverse of the Poisson 
integral of a delta measure on 1).
As far as we know the study of the stable rank of $\N_+$ is wide open.

\subsection{Corona problem in quotient algebras}\label{S-wep}

As mentioned before, the class $\N_+$ is an algebra where the invertible 
functions
are exactly the outer functions. So any quotient of the Smirnov class by a 
principal
ideal (which are the only closed ideals \cite[Theorem 2]{RS}) can be 
represented by 
$\N^+_I:=\N_+/I\N_+$, where $I$ is  an inner function.

Carefully following the proof of Theorems~\ref{corona_bound} and \ref{finite_u} 
above, 
we see that the natural analogues hold for $\N_+$, with the usual replacement 
of $\Har_+(\D)$ by $QB(\D)$.

This explains the situation for quotients of type $\N^+_B$, with $B$ Blaschke 
product, but not 
the general situation of quotients by inner functions. Here
we say that an inner function $I$ with zero set $Z$ satisfies the Smirnov WEP 
if and only if for any $H_1\in QB_+(\D)$ there exists 
$H_2\in QB_+(\D)$ such that $|B(z)|\geq e^{-H_2(z)}$ for any $z\in\D$ such that 
$\rho(z,Z)\geq e^{-H_1(z)}$.

Notice that if it does and $I= BS$, where $B$ is a Blaschke product
and $S$ a singular inner function, then since $|I|\le |B|$ and 
$I^{-1}\{0\}=B^{-1}\{0\}$, 
it is clear that  $B$ must be Smirnov WEP as well.  One may wonder which 
singular inner functions are 
admissible as divisors of Smirnov WEP functions, in analogy to the study begun 
in \cite{Bo1}, \cite{BNT}.
It turns out that there aren't any besides the constants.

\begin{theorem}
\label{smirnovwep}
Let $I=BS$ be an inner function, where $B$ is a Blaschke product
and $S$ is singular inner. Then $I$ satisfies the Smirnov WEP if and only if 
$S=e^{i\theta}$ (a unimodular constant)
and $B^{-1}\{0\}$
is a finite union of Smirnov interpolating sequences.

As a consequence, if $I$ is an inner function, the Corona Property holds for 
$\N^+_I$
if and only if $I$ is a Blaschke product with its zero set being a  finite 
union of Smirnov interpolating sequences.
\end{theorem}

To see this one has to check that the WEP condition forces $-\log |S|$ to have 
a quasi-bounded harmonic majorant,
which is only possible when $S$ is a constant.

\end{document}